\documentclass{amsart}
\usepackage{amssymb,amsmath}
\usepackage{stmaryrd}
\usepackage{yhmath}
\usepackage{latexsym}
\usepackage{amsthm}
\usepackage{amscd}
\usepackage{mathrsfs}
\usepackage[all]{xy}
\usepackage{mathtools}
\usepackage{arydshln}
\usepackage{bbm}
\usepackage{url}

\newcommand{\tu}{\mathrm{tu}}
\newcommand{\tn}{\mathrm{tn}}

\newcommand{\mcO}{\mathcal{O}}

\newcommand{\tPhi}{\widetilde{\Phi}}

\newcommand{\w}{\widetilde}
\renewcommand{\t}{\tilde}
\newcommand{\Z}{\mathbb{Z}}

\newcommand{\R}{\mathbb{R}}
\newcommand{\Q}{\mathbb{Q}}
\newcommand{\C}{\mathbb{C}}

\newcommand{\G}{\mathbf{G}}
\newcommand{\J}{\mathbf{J}}
\newcommand{\B}{\mathcal{B}}
\newcommand{\wG}{\widetilde{\mathbf{G}}}
\newcommand{\bfH}{\mathbf{H}}
\newcommand{\mfc}{\mathfrak{c}}
\newcommand{\mfg}{\mathfrak{g}}
\newcommand{\mfh}{\mathfrak{h}}
\newcommand{\mfi}{\mathfrak{i}}
\newcommand{\mfj}{\mathfrak{j}}

\DeclareMathOperator{\tr}{tr}
\DeclareMathOperator{\tc}{tc}
\DeclareMathOperator{\depth}{depth}

\DeclareMathOperator{\vol}{vol}
\DeclareMathOperator{\GL}{GL}
\DeclareMathOperator{\SO}{SO}
\DeclareMathOperator{\SL}{SL}
\DeclareMathOperator{\Sp}{Sp}

\DeclareMathOperator{\U}{U}
\DeclareMathOperator{\Lie}{Lie}
\DeclareMathOperator{\Nil}{Nil}

\DeclareMathOperator{\nInd}{n-Ind}

\DeclareMathOperator{\End}{End}

\DeclareMathOperator{\Ad}{Ad}

\DeclareMathOperator{\diag}{diag}
\DeclareMathOperator{\Span}{Span}
\DeclareMathOperator{\Out}{Out}

\theoremstyle{plain}
\newtheorem{thm}{Theorem}[section]
\newtheorem*{thm*}{Theorem}
\newtheorem{prop}[thm]{Proposition}
\newtheorem{lem}[thm]{Lemma}
\newtheorem{cor}[thm]{Corollary}

\newtheorem{hyp}[thm]{Hypothesis}

\theoremstyle{definition}
\newtheorem{defn}[thm]{Definition}

\theoremstyle{remark}
\newtheorem{rem}[thm]{Remark}
\newtheorem*{claim*}{Claim}

\SelectTips{cm}{11}

\title{Depth preserving property of the local Langlands correspondence for quasi-split classical groups in a large residual characteristic}
\author{Masao Oi}
\address{Graduate School of Mathematical Sciences, 
the University of Tokyo, 3-8-1 Komaba, Meguro-ku, Tokyo 153-8914, Japan.}
\email{masaooi@ms.u-tokyo.ac.jp}

\begin{document}

\begin{abstract}
For a quasi-split classical group over a $p$-adic field with sufficiently large residual characteristic, we prove that the maximum of depth of representations in each $L$-packet equals the depth of the corresponding $L$-parameter.
Furthermore, for quasi-split unitary groups, we show that the depth is constant in each $L$-packet.
The key is an analysis of the endoscopic character relation via harmonic analysis based on the Bruhat--Tits theory.
These results are slight generalizations of a result of Ganapathy and Varma in \cite{MR3709003}.
\end{abstract}


\maketitle

\section{Introduction}
Let $F$ be a $p$-adic field and $\G$ either a general linear group or a quasi-split classical group over $F$.
We denote the set of equivalence classes of irreducible smooth representations of $G=\G(F)$ by $\Pi(\G)$, and the set of conjugacy classes of $L$-parameters of $\G$ by $\Phi(\G)$.
Then the \textit{local Langlands correspondence for $\G$}, which has been established by Harris--Taylor (general linear groups, \cite{MR1876802}),  Arthur (symplectic or orthogonal groups, \cite{MR3135650}), and Mok (unitary groups, \cite{MR3338302}) gives a natural map from the set $\Pi(\G)$ to the set $\Phi(\G)$ with finite fibers (called $L$-packets).
In other words, the local Langlands correspondence gives a natural partition of the set $\Pi(\G)$ into finite sets parametrized by $L$-parameters:
\[
\Pi(\G)=\bigsqcup_{\phi\in\Phi(\G)} \Pi^{\G}_{\phi}.
\]

It is known that the local Langlands correspondence satisfies a lot of natural properties beyond its characterization.
One example for such a phenomenon is the \textit{depth preserving} property of the local Langlands correspondence for general linear groups.
To be more precise, let $\GL_{N}$ be the general linear group of size $N$.
Recall that every irreducible smooth representation of $\GL_{N}(F)$ has its \textit{depth}, which is a numerical invariant (non-negative rational number) defined by the theory of Moy--Prasad filtrations (\cite{MR1371680}).
Roughly speaking, the depth of a representation express how large subgroups having an invariant part in the representation are (see Definition \ref{defn:repdepth} for the precise definition).
On the other hand, also for an $L$-parameter of $\GL_{N}$, we can define its \textit{depth} by using the upper ramification filtration of the Weil group $W_{F}$ of $F$ (see Definition \ref{defn:pardepth} for the precise definition).
The depth of an $L$-parameter measures how deep the ramification of the $L$-parameter is.
Then it is known that the local Langlands correspondence for $\GL_{N}$ preserves the depth (see, e.g., \cite{MR3579297}).
Note that, when $N=1$, this is nothing but the well-known property of the local class field theory about the correspondence between higher unit groups of $F^{\times}$ and the upper ramification filtration of $W_{F}^{\mathrm{ab}}$.

Therefore it is a natural attempt to investigate the relationship between the depth of representations and that of $L$-parameters under the local Langlands correspondence for other groups.
At present, there is no complete description of the behavior of the depth under the local Langlands correspondence for general groups except for some small groups (see, for example, \cite{MR3618046} for the details).
However, in a recent paper \cite{MR3709003}, Ganapathy and Varma give the following partial answer to this problem:
\begin{thm}[{\cite[Corollary 10.6.4]{MR3709003}}]\label{thm:GV}
Let $\bfH$ be a quasi-split symplectic or special orthogonal group over $F$.
We assume that the residual characteristic is large enough.
Let $\phi$ be a tempered $L$-parameter of $\bfH$, and $\Pi_{\phi}^{\bfH}$ the $L$-packet of $\bfH$ for $\phi$.
Then we have
\[
\max\bigl\{\depth(\pi) \,\big\vert\, \pi \in \Pi_{\phi}^{\bfH}\bigr\}
\leq
\depth(\phi).
\]
\end{thm}

Our main theorem in this paper is the following:
\begin{thm}[Theorem \ref{thm:main-non-temp}]\label{thm:mainintro}
Let $\bfH$ be a quasi-split classical $($namely, symplectic, special orthogonal, or unitary$)$ group over $F$.
We assume that the residual characteristic is large enough $($see Hypothesis $\ref{hyp}$ for the detail$)$.
Let $\phi$ be an $L$-parameter of $\bfH$, and $\Pi_{\phi}^{\bfH}$ the $L$-packet of $\bfH$ for $\phi$.
Then we have
\[
\max\bigl\{\depth(\pi) \,\big\vert\, \pi \in \Pi_{\phi}^{\bfH}\bigr\}
=
\depth(\phi).
\]
\end{thm}

From now on, let $\bfH$ be a quasi-split classical group over $F$.
Before we explain the sketch of our proof, we recall the \textit{endoscopic character relation}, which is used to formulate the naturality of the local Langlands correspondence for $\bfH$.
First, we can regard $\bfH$ as an \textit{endoscopic group} of a (twisted) general linear group $\GL_{N}$ over $F$ (strictly speaking, when $\bfH$ is a unitary group, we have to consider the Weil restriction of $\GL_{N}$ with respect to a quadratic extension associated to $\bfH$).
In particular, we have an embedding $\iota$ from the $L$-group of $\bfH$ to that of $\GL_{N}$.
Here the size $N$ of the general linear group depends on each classical group (see Section \ref{sec:pre}.1 for details).
Now let us take an $L$-parameter $\phi$ of $\bfH$.
By the theory of Langlands classification, we can extend the local Langlands correspondence for tempered representations to nontempered representations formally.
Therefore, to consider the naturality of the local Langlands correspondence, we may assume that $\phi$ is tempered.
Then, by noting that $\phi$ is a homomorphism from $W_{F}\times\SL_{2}(\C)$ to ${}^{L}\bfH$, we obtain an $L$-parameter of $\GL_{N}$ by composing $\phi$ with the embedding $\iota$.
From these $L$-parameters, we get representations of two different groups.
One is the representation $\pi_{\phi}^{\GL_{N}}$ of $\GL_{N}(F)$ corresponding to $\iota\circ\phi$ under the local Langlands correspondence for $\GL_{N}$ (note that, for $\GL_{N}$, each $L$-packet is a singleton).
The other is an $L$-packet $\Pi_{\phi}^{\bfH}$, which is a finite set of representations of $H$, corresponding to $\phi$ under the local Langlands correspondence for $\bfH$.
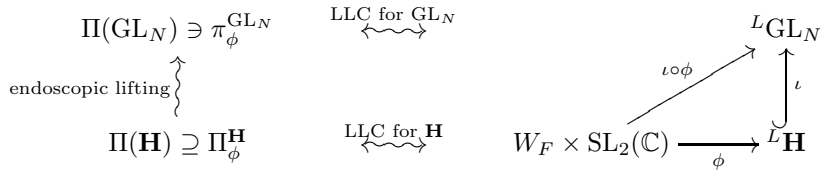
\[
\xymatrix{
\Pi(\GL_{N}) \ni \pi_{\phi}^{\GL_{N}}& \ar@{<~>}[r]^-{\text{LLC for $\GL_{N}$}} &&& {}^{L}\!\GL_{N}\\
\Pi(\bfH) \supseteq \Pi_{\phi}^{\bfH}\ar@{~>}[u]^-{\text{endoscopic lifting}}& \ar@{<~>}[r]^-{\text{LLC for $\bfH$}} &&W_F\times\SL_2(\C) \ar[r]_-{\phi} \ar[ru]^-{\iota\circ\phi} & {}^{L}\bfH \ar@{^{(}->}[u]_-{\iota} \\
}
\]
In this situation, we say that $\pi_{\phi}^{\GL_{N}}$ is the \textit{endoscopic lift} of $\Pi_{\phi}^{\bfH}$ from $\bfH$ to $\GL_{N}$.
Then the endoscopic character relation is the following equality satisfied by the \textit{twisted character} $\Theta_{\phi,\theta}^{\GL_{N}}$ of $\pi_{\phi}^{\GL_{N}}$ and the \textit{characters} $\Theta_{\pi}$ of representations $\pi$ belonging to $\Pi_{\phi}^{\bfH}$:
\[
\Theta_{\phi,\theta}^{\GL_{N}}(f)
=\sum_{\pi\in\Pi_{\phi}^{\G}}\Theta_{\pi}(f^{H}).
\]
Here $f$ is any test function of $\GL_{N}(F)$ and $f^{H}$ is its Langlands--Shelstad--Kottwitz transfer to $H$ (see Section \ref{sec:pre}.2 for the details).
The important point is that the composition with the $L$-embedding does not change the depth of $L$-parameters.
Namely, by this formulation of the naturality of the local Langlands correspondence for $\bfH$ and the depth preserving property of the local Langlands correspondence for $\GL_{N}$, the depth preserving problem of the local Langlands correspondence for $\bfH$ is equivalent to that of the endoscopic lifting from $\bfH$ to $\GL_{N}$.
We tackle the latter problem by investigating the endoscopic character relations via harmonic analysis on $p$-adic reductive groups.

To explain the strategy of our proof of Theorem \ref{thm:mainintro}, we recall Ganapathy--Varma's method used in the proof of Theorem \ref{thm:GV} in \cite{MR3709003}.
The key tools in their proof are the following DeBacker's two results:
\begin{enumerate}
\item
Description of the radii of the character expansions of irreducible smooth representations (``homogeneity'', established in \cite{MR1914003}).
\item
Parametrization of nilpotent orbits via Bruhat--Tits theory (established in \cite{MR1935848}).
\end{enumerate}

Let us recall them.
First, for every irreducible smooth representation $\pi$ of $H$, we have its \textit{character} $\Theta_{\pi}$, which is an invariant distribution on $H$.
In general, it is very complicated and difficult to describe the behavior of the character $\Theta_{\pi}$. 
However, in some ``small neighborhood'' of the origin, we can express the character $\Theta_{\pi}$ as a linear combination of the nilpotent orbital integrals of Fourier transforms.
More precisely, if we have an appropriate exponential map $\mfc_{\bfH}$ from the Lie algebra $\mfh$ to $H$, then, for every function $f$ on the Lie algebra supported in the ``small neighborhood'' of the origin, we have
\[
\Theta_{\pi}(f\circ\mfc_{\bfH}^{-1})
=
\sum_{\mathcal{O}\in\Nil(\mfh)}c_{\mathcal{O}}\cdot\widehat{\mu_{\mathcal{O}}}(f)
\]
(this is called the \textit{character expansion} of the character of a representation, and established by Harish-Chandra (\cite{MR1702257})).
Then the following question about this character expansion naturally arises: what is the optimal size of the ``small neighborhood''?
In \cite{MR1914003}, DeBacker gave an answer to this question by using the Bruhat--Tits theory.
To be more precise, we put $r$ to be the depth of an irreducible smooth representation $\pi$ and $H_{r+}$ to be the union of $(r+)$-th Moy--Prasad filtrations of parahoric subgroups (see Section \ref{sec:pre}.5 for details).
Then DeBacker proved that the character expansion is valid on $\mfc_{\bfH}^{-1}(H_{r+})$ under some assumptions on the residual characteristic $p$ (see Section \ref{sec:max}.2 for the detail of the assumption). 
On the other hand, in another paper \cite{MR1935848}, DeBacker established a parametrization of nilpotent orbits via Bruhat--Tits theory under some assumptions on the residual characteristic.
By using this parametrization, we can recover the depth of an irreducible smooth representation from the radius of its character expansion.
Namely we can show that if $\Theta_{\pi}$ has a character expansion on $\mfc_{\bfH}^{-1}(H_{s+})$ for some positive number $s\in\R$, then the depth of $\pi$ is not greater than $s$.
In other words, we can say that the depth of an irreducible smooth representation gives an optimal radius of the character expansion.

On the other hand, for twisted characters of irreducible smooth representations, the theory of the character expansion can be formulated as follows:
for every function $f$ on the Lie algebra of $\GL_{N}$ supported in the ``small neighborhood'' of the origin, we have
\[
\Theta_{\phi,\theta}^{\GL_{N}}(f\circ\mfc^{-1})
=
\sum_{\mathcal{O}\in\Nil(\mfg_{\theta})}c_{\mathcal{O}}\cdot\widehat{\mu_{\mathcal{O}}}(f_{\theta}).
\]
Here $\mfc$ is a kind of exponential map (see Section \ref{sec:pre}.4), $\mfg_{\theta}$ is the Lie algebra of the group $\G_{\theta}$ which is the fixed part of an involution $\theta$ of $\G$ (see Section \ref{sec:pre}.1), and $f_{\theta}$ is a function on $\mfg_{\theta}$ which is a \textit{semisimple descent} of $f$ (see Section \ref{sec:descent}).
For this expansion of twisted characters, in \cite{MR2306039}, Adler and Korman established a result which is analogous to that of DeBacker under some assumptions on the residual characteristic of the same type as DeBacker's one.
Namely, they described the size of the ``small neighborhood'' where the character expansion is valid in terms of the depth of the representations.

Now we recall Ganapathy--Varma's method.
Their idea is to compare the depth of a tempered $L$-packet $\Pi_{\phi}^{\bfH}$ and its endoscopic lift $\pi_{\phi}^{\GL_{N}}$ by comparing the radii of the character expansions for $\Pi_{\phi}^{\bfH}$ and $\pi_{\phi}^{\GL_{N}}$ via the endoscopic character relation.
Under the assumption that the residual characteristic is large enough to satisfy the assumptions of DeBacker's results and Adler--Korman's result, Ganapathy and Varma proved Theorem \ref{thm:GV} in the following way:
\begin{enumerate}
\item
The radius of the character expansion of the twisted character $\Theta_{\phi,\theta}^{\GL_{N}}$ of $\pi_{\phi}^{\GL_{N}}$ is given by $\depth(\pi_{\phi}^{\GL_{N}})+$ (Adler--Korman's result).
\item
By using the endoscopic character relation, we know that the maximum of the radii of the character expansions of the characters of representations belonging to $\Pi_{\phi}^{\bfH}$ is smaller than $\depth(\pi_{\phi}^{\GL_{N}})+$.
\item
By using DeBacker's parametrization of the nilpotent orbits, we can conclude that the maximum of the depth of representations belonging to $\Pi_{\phi}^{\bfH}$ is smaller than $\depth(\pi_{\phi}^{\GL_{N}})+$.
\end{enumerate}

Then it is natural to consider the converse direction of this argument by swapping the roles of $\GL_{N}$ and $\bfH$, that is:
\begin{enumerate}
\item[$(1)'$]
The maximum of radii of the character expansions of the characters of representations $\pi$ belonging to $\Pi_{\phi}^{\bfH}$ is given by the maximum of $\depth(\pi)+$ (DeBacker's result).
\item[$(2)'$]
By using the endoscopic character relation, we know that the radii of the character expansions of the twisted characters $\Theta_{\phi,\theta}^{\GL_{N}}$ of $\pi_{\phi}^{\GL_{N}}$ is smaller than $\max\{\depth(\pi)+\}$.
\item[$(3)'$]
By using DeBacker's parametrization of the nilpotent orbits, we conclude that the depth of $\pi_{\phi}^{\GL_{N}}$ is smaller than $\max\{\depth(\pi)+\}$.
\end{enumerate}
However, we can not so immediately imitate Ganapathy--Varma's arguments.
The problem is in the step $(3)'$. 
That is, the behavior of the characteristic functions of the Moy--Prasad filtrations of parahoric subgroups under the semisimple descent is not so clear.

In this paper, in order to carry out the step $(3)'$, we investigate the semisimple descents for the characteristic functions of the Moy--Prasad filtrations of parahoric subgroups of general linear groups by a group-theoretic computation.
Then, as a consequence of such a computation, we can complete the above arguments of the converse direction and get the following converse inequality:
\[
\max\bigl\{\depth(\pi) \,\big\vert\, \pi \in \Pi_{\phi}^{\bfH}\bigr\}
\geq
\depth(\phi).
\]
In particular, by combining this with Theorem \ref{thm:GV}, we get the equality (Theorem \ref{thm:mainintro}).

When $\bfH$ is a unitary group $\U_{E/F}(N)$ associated to a quadratic extension $E$ of $F$, the semisimple descent coincides with the Langlands--Shelstad--Kottwitz transfer.
Thus, by the above computation of the semisimple descents for the characteristic functions of the Moy--Prasad filtrations, we get the following generalization of the \textit{fundamental lemma} to positive depth direction:
\begin{thm}[Theorem \ref{thm:MPFL}]\label{thm:MPFLintro}
We assume that the residual characteristic $p$ is not equal to $2$.
We take a point $x$ of the Bruhat--Tits building of $\bfH$ and we identify it with a point of the Bruhat--Tits building of $\G=\mathrm{Res_{E/F}}\GL_{N}$ canonically.
Let $r\in\R_{>0}$.
Let $H_{x,r}$ and $G_{x,r}$ be the $r$-th Moy--Prasad filtrations with respect to the point $x$.
Then $\vol(H_{x,r})^{-1}\mathbbm{1}_{H_{x,r}}\in C_{c}^{\infty}(H)$ is a transfer of $\vol(G_{x,r})^{-1}\mathbbm{1}_{G_{x,r}}\in C_{c}^{\infty}(G)$.
\end{thm}

We remark that a similar assertion for $r=0$ (namely, the fundamental lemma for parahoric subgroups) in the case where $E$ is unramified over $F$ are proved in \cite{MR868140} (see also \cite{MR2553880}).
This theorem is not only interesting itself, but also having an application to the depth preserving problem of the endoscopic lifting.
We can immediately deduce the following theorem from Theorem \ref{thm:MPFLintro} by using the endoscopic character relation:
\begin{thm}[Theorem \ref{thm:main2-non-temp}]
We assume that the residual characteristic $p$ is not equal to $2$.
Let $\phi$ be an $L$-parameter of $\bfH$, and $\Pi_{\phi}^{\bfH}$ the $L$-packet of $\bfH$ for $\phi$.
Then we have
\[
\min\bigl\{\depth(\pi) \,\big\vert\, \pi \in \Pi_{\phi}^{\bfH}\bigr\}
\geq
\depth(\phi).
\]
In particular, by combining it with Theorem $\ref{thm:GV}$, we have
\[
\depth(\pi)=\depth(\phi)
\]
for every $\pi\in\Pi_{\phi}^{\bfH}$ under the assumption that the residual characteristic $p$ is large enough $($see Hypothesis $\ref{hyp}$ for the detail$)$.
\end{thm}

We finally remark that we cannot expect that the inequality in Theorem \ref{thm:GV}  holds for a general connected reductive group.
For example, in \cite[Section 7.4]{MR3164986} Reeder and Yu constructed a candidate of the $L$-parameters corresponding to ``simple supercuspidal representations'' of $\mathrm{SU}_{p}(\Q_{p})$ for an odd prime $p$, by assuming Hiraga--Ichino--Ikeda's formal degree conjecture.
In this example, the depth of simple supercuspidal representations and the depth of their $L$-parameters are given by $\frac{1}{2p}$ and $\frac{1}{2(p+1)}$, respectively.
See also \cite[Section 3.3]{MR3618046}.

We explain on the organization of this paper.
In Section \ref{sec:pre}, we collect some basic preliminaries which will be needed in this paper.
In Section \ref{sec:descent}, we compute the semisimple descent for the Moy--Prasad filtrations of parahoric subgroups of general linear groups.
In Section \ref{sec:max}, by using the results in Section \ref{sec:descent}, we evaluate the maximum of the depth of representations in each $L$-packet of classical groups according to the converse of Ganapathy--Varma's method.
In Section \ref{sec:min-unitary}, by focusing on the unitary case, we evaluate the minimum of the depth of representations in each $L$-packet.

\medbreak
\noindent{\bfseries Acknowledgment.}\quad
The author would like to thank his advisor Yoichi Mieda for his support and encouragement.
This work was carried out with the support from the Program for Leading Graduate Schools, MEXT, Japan.
This work was also supported by JSPS Research Fellowship for Young Scientists and KAKENHI Grant Number 17J05451.

\setcounter{tocdepth}{2}
\tableofcontents

\medbreak
\noindent{\bfseries Notation.}\quad
Let $p$ be a prime number.
In this paper, we always assume that $p$ is not equal to $2$.
We fix a $p$-adic field $F$.
We denote the Weil group of $F$ by $W_{F}$.
For an algebraic variety $\J$ over $F$, we denote its $F$-valued points by $J$.
When $\J$ is a connected reductive group, we write $\widehat{\J}$ and ${}^{L}\J=\widehat{\J}\rtimes W_{F}$ for its Langlands dual group and $L$-group, respectively.

\section{Basic preliminaries}\label{sec:pre}

\subsection{Twisted endoscopy for general linear groups}
Let $(\G, \bfH)$ be one of the following pairs of connected reductive groups over $F$:
\begin{itemize}
\item[(1)]
$\G:=\GL_{2n+1}$ and $\bfH$ is the split symplectic group of size $2n$.
\item[(2)]
$\G:=\GL_{2n}$ and $\bfH$ is the split special orthogonal group of size $2n+1$.
\item[(3)]
$\G:=\GL_{2n}$ and $\bfH$ is a quasi-split special orthogonal group of size $2n$.
\item[(4)]
$\G:=\mathrm{Res}_{E/F}\GL_{N}$ for a quadratic extension $E/F$, and $\bfH$ is the quasi-split unitary group with respect to $E/F$ in $N$ variables.
\end{itemize}

For $\G$, we consider the following automorphism over $F$:
\[
\theta\colon\G\rightarrow\G;\quad
g\mapsto J_{N} {}^{t}\!c(g)^{-1}J_{N}^{-1}.
\]
Here $c$ is 
\[
\begin{cases}
\text{the identity map} & \text{if } \G=\GL_{N},\\
\text{the Galois conjugation of $E/F$} & \text{if } \G=\mathrm{Res}_{E/F}\GL_{N},
\end{cases}
\]
and $J_{N}$ is the anti-diagonal matrix whose $(i, N+1-i)$-th entry is given by $(-1)^{i-1}$.
Then $\bfH$ is an endoscopic group for $(\G,\theta)$.

On the other hand, by using the above automorphism $\theta$, we can define a disconnected reductive group $\G\rtimes\langle\theta\rangle$ as the semi-direct product of $\G$ and the group $\langle\theta\rangle$ generated by $\theta$.
We write $\wG$ for the connected component $\G\rtimes\theta$ of this group which does not contain the unit element.
We note that $\wG$ has left and right actions of $\G$ and is a bi-$\G$-torsor with respect to these actions.
Namely, $(\wG,\G)$ is a ``twisted space'' in the sense of Labesse.

We denote the identity component $(\G^{\theta})^{0}$ of the $\theta$-fixed part of $\G$ by $\G_{\theta}$.
Note that this is given by
\[
\G_{\theta}
=
\begin{cases}
\SO_{2n+1} & \text{if } \G=\GL_{2n+1},\\
\Sp_{2n} & \text{if } \G=\GL_{2n},\\
\U_{E/F}(N) & \text{if } \G=\mathrm{Res}_{E/F}\GL_{N}.
\end{cases}
\]

\subsection{Orbital integral and the Langlands--Shelstad--Kottwitz transfer}
Let $\J$ be a connected reductive group over $F$.
For an open subset $\mathcal{V}$ of $J$ which is invariant under $J$-conjugation, we denote the set of conjugacy classes of strongly regular semisimple elements of $J$ belonging to $\mathcal{V}$ by $\Gamma(\mathcal{V})$.
For an element $f\in C_{c}^{\infty}(\mathcal{V})$ and $\gamma\in\Gamma(\mathcal{V})$, we define the \textit{normalized orbital integral} of $f$ at $\gamma$ by
\[
I_{\gamma}(f)
:=
|D_{J}(\gamma)|^{\frac{1}{2}}\cdot
\int_{J_{\gamma}\backslash J}
f(g^{-1}\gamma g) \,d\dot{g},
\]
where $D_{J}(\gamma)$ is the Weyl discriminant of $\gamma$ in $J$, $J_{\gamma}$ is the $F$-valued points of the centralizer $\J_{\gamma}$ of $\gamma$ in $\J$, and $d\dot{g}$ is a right $J$-invariant measure on $J_{\gamma}\backslash J$ induced by Haar measures on $J$ and $J_{\gamma}$.
Then we can regard $I(f)$ as a function on $\Gamma(\mathcal{V})$.
We denote the set of such normalized orbital integrals by $\mathcal{I}(\mathcal{V})$:
\[
\mathcal{I}(\mathcal{V})
:=
\{I(f) \mid f\in C_{c}^{\infty}(\mathcal{V})\}
\subset
\{\text{$\C$-valued functions on $\Gamma(\mathcal{V})$}\}.
\]

Now we furthermore assume that $\mathcal{V}$ is invariant under stable conjugacy.
Then, for $f\in C_{c}^{\infty}(\mathcal{V})$ and a strongly regular semisimple element $\gamma\in \Gamma(\mathcal{V})$, we can define the \textit{stable orbital integral} of $f$ at $\gamma$ by
\[
SI_{\gamma}(f):=
\sum_{\gamma'\sim_{\mathrm{st}}\gamma / \sim}
I_{\gamma'}(f),
\]
where the sum is over the set of $J$-conjugacy classes of stable conjugacy classes of $\gamma$.
If we put $\mathcal{SI}(\mathcal{V})$ to be the set of such stable orbital integrals, then we have a canonical surjection $\mathcal{I}(\mathcal{V})\twoheadrightarrow\mathcal{SI}(\mathcal{V});\, I(f)\mapsto SI(f)$.

We note that every $J$-invariant distribution on $C_{c}^{\infty}(\mathcal{V})$ (that is, a $\C$-linear functional $C_{c}^{\infty}(\mathcal{V})\rightarrow\C$) factors through this space $\mathcal{I}(\mathcal{V})$.
If a $J$-invariant distribution on $C_{c}^{\infty}(\mathcal{V})$ factors through $\mathcal{SI}(\mathcal{V})$, we say that it is a \textit{stable distribution}.

For the twisted space $\w{\G}$, we can define similar objects.
Namely, for an open subset $\w{\mathcal{V}}$ of $\w{G}$ which is invariant under $G$-conjugation, we denote the set of conjugacy classes of strongly regular semisimple elements of $\w{\mathcal{V}}$ belonging to $\w{\mathcal{V}}$ by $\Gamma(\w{\mathcal{V}})$.
For an element $f\in C_{c}^{\infty}(\w{G})$ and $\t\delta\in\Gamma(\w{G})$, we define the \textit{normalized orbital integral} of $f$ at $\t\delta$ by
\[
I_{\t\delta}(f)
:=
|D_{\w{G}}(\t\delta)|^{\frac{1}{2}}\cdot
\int_{G_{\t\delta}\backslash G}
f(g^{-1}\t\delta g) \,d\dot{g},
\]
where $D_{\w{G}}(\t\delta)$ is the Weyl discriminant of $\t\delta$ in $\w{G}$, $G_{\t\delta}$ is the $F$-valued points of the centralizer $\G_{\t\delta}$ of $\t\delta$ in $\G$, and $d\dot{g}$ is a right $G$-invariant measure on $G_{\t\delta}\backslash G$ induced by Haar measures on $G$ and $G_{\t\delta}$.
We denote the set of such normalized orbital integrals by $\mathcal{I}(\w{\mathcal{V}})$:
\[
\mathcal{I}(\w{\mathcal{V}})
:=
\{I(f) \mid f\in C_{c}^{\infty}(\w{\mathcal{V}})\}
\subset
\{\text{$\C$-valued functions on $\Gamma(\w{\mathcal{V}})$}\}
\]
(note that every $G$-invariant distribution on $C_{c}^{\infty}(\w{\mathcal{V}})$ factors through this space $\mathcal{I}(\w{\mathcal{V}})$).

Now we can define the notion of \textit{transfer} for test functions.
Let $(\G,\bfH)$ be a pair as in the previous subsection.
We denote by $\Delta^{\mathrm{IV}}$ the Kottwitz--Shelstad transfer factor with respect to $(\G,\bfH)$ without the fourth factor $\Delta_{\mathrm{IV}}$ (see \cite[Section 5]{MR1687096} for the definition).
Note that, in order to normalize $\Delta^{\mathrm{IV}}$, we have to choose  a $\theta$-stable Whittaker data of $\G$.
From now on, we fix a $\theta$-stable Whittaker data of $\G$.

\begin{defn}[matching orbital integral, transfer of test functions]\label{defn:moi}
We say $f\in C_{c}^{\infty}(\w{G})$ and $f^{H}\in C_{c}^{\infty}(H)$ have \textit{matching orbital integrals} if, for every strongly $\G$-regular semisimple element $\gamma$ of $H$, we have
\[
SI_{\gamma}(f^{H})
=
\sum_{\t\delta\leftrightarrow\gamma/\sim} \Delta^{\mathrm{IV}}(\gamma,\t\delta)I_{\t\delta}(f),
\]
where the sum is over the set of $G$-conjugacy classes of strongly regular elements of $\w{G}$ such that $\gamma$ is a norm of $\t\delta$.
In this situation, we say that $f^{H}$ is a \textit{transfer} of $f$.
\end{defn}

Here we note that, we choose measures appearing in the above orbital integrals as in the manner of \cite[Section 5.5]{MR1687096}.
See also \cite[Remark 6.6.2]{MR3709003}.

On the existence of transfer of test functions, we have the following highly non-trivial theorem, which was established by a great deal of efforts of a lot of people represented by Waldspurger, Ng\^o, and so on (see, e.g., \cite[54 page]{MR3135650} for the details):
\begin{thm}[Langlands--Shelstad--Kottwitz's transfer conjecture]\label{thm:LSK}
For every $f\in C_{c}^{\infty}(\w{G})$, there exists a transfer $f^{H}\in C_{c}^{\infty}(H)$ of $f$.
Namely, we have a map $\mathcal{I}(\w{G})\rightarrow \mathcal{SI}(H)$ characterized by the matching orbital integral condition.
\end{thm}

\subsection{Arthur's theory and the endoscopic character relation}
In this section, we recall Arthur's theory of the endoscopic classification of representations of classical groups over $F$.

To state Arthur's theorem on the classification of representations, we define some notations.
Let $\J$ be either a general linear group or a quasi-split classical group over $F$.
We write $\Out(\J)$ for the group of outer automorphisms of $\J$, namely the quotient of the group $\mathrm{Aut}(\J)$ of automorphisms of $\J$ by the group $\mathrm{Inn}(\J)$ of inner automorphism of $\J$.
Here we note that $\Out(\J)$ is non-trivial only when $\J$ is an even special orthogonal group, and that in this case we have $\Out(\J)\cong\Z/2\Z$.

We denote the set of equivalence classes of irreducible smooth (resp.\ tempered) representations of $J$ by $\Pi(\J)$ (resp.\ $\Pi_{\mathrm{temp}}(\J)$).
Then the group $\Out(\J)$ acts on these sets.
We write $\widetilde{\Pi}(\J)$ (resp.\ $\widetilde{\Pi}_{\mathrm{temp}}(\J)$) for the set of $\Out(\J)$-orbits in $\Pi(\J)$ (resp.\ $\Pi_{\mathrm{temp}}(\J)$).
For an $\Out(\J)$-orbit $\tilde{\pi}$ in $\w\Pi_{\mathrm{temp}}(\J)$, we put
\[
\Theta_{\tilde{\pi}}
:=
\frac{1}{|\t\pi|}\sum_{\pi\in\t\pi}\Theta_{\pi},
\]
where $\Theta_{\pi}$ is the character of $\pi$.

We denote the set of $\widehat{\J}$-conjugacy classes of $L$-parameters (resp.\ tempered $L$-parameters) of $\J$ by $\Phi(\J)$ (resp.\ $\Phi_{\mathrm{temp}}(\J)$), and their $\Out(\J)$-orbits by $\tPhi(\J)$ (resp.\ $\tPhi_{\mathrm{temp}}(\J)$).

Now let $(\G,\bfH)$ be one of the pairs considered in Section \ref{sec:pre}.1.
Then, since $\bfH$ is an endoscopic group of $\G$, we can regard an element $\phi\in\Phi_{\mathrm{temp}}(\bfH)$ as an $L$-parameter of $\G$.
This operation induces an injection from $\tPhi_{\mathrm{temp}}(\bfH)$ to $\Phi_{\mathrm{temp}}(\G)$.

The following is the local part of Arthur's theory (the \textit{local Langlands correspondence} for $\bfH$):
\begin{thm}[{\cite[Theorems 1.5.1 and 2.2.1]{MR3135650}}, {\cite[Theorems 2.5.1 and 3.2.1]{MR3338302}}]\label{thm:Arthur}
We have a natural partition
\[
\widetilde{\Pi}_{\mathrm{temp}}(\bfH) = \bigsqcup_{\phi\in\tPhi_{\mathrm{temp}}(\bfH)} \Pi^{\bfH}_{\phi}.
\]
of $\widetilde{\Pi}_{\mathrm{temp}}(\bfH)$ into finite subsets $($which are called the $L$-packets$)$ parametrized by tempered $L$-parameters.
Here each $L$-packet $\Pi_{\phi}^{\bfH}$ satisfies the stability $($namely, the sum $\Theta_{\phi}^{\bfH}$ of the characters of representations belonging to $\Pi_{\phi}^{\bfH}$ is a stable distribution$)$ and the following equality for every $f\in C_{c}^{\infty}(\w{G})$:
\[
\Theta_{\phi,\theta}^{\G}(f)
=
\Theta_{\phi}^{\bfH}(f^{H}),
\]
where 
\begin{itemize}
\item
$\Theta_{\phi,\theta}^{\G}$ is the $\theta$-twisted character of $\pi_{\phi}^{\G}$, which is the irreducible tempered representation of $G$ corresponding to $\phi$ under the local Langlands correspondence for $\G$, normalized by the fixed $\theta$-stable Whittaker data of $\G$,
\item
$\Theta_{\phi}^{\bfH}$ is the sum of the characters of representations belonging to $\Pi_{\phi}^{\bfH}$, and
\item
$f^{H}\in C_{c}^{\infty}(H)$ is a transfer of $f$ $($note that the existence of a transfer is assured by Theorem $\ref{thm:LSK}$$)$.
\end{itemize}
\end{thm}
We call the representation $\pi_{\phi}^{\G}$ the \textit{endoscopic lift} of the $L$-packet $\Pi_{\phi}^{\bfH}$ from $\bfH$ to $\G$, and the equality in the above theorem the \textit{endoscopic character relation} for $\pi_{\phi}^{\G}$ and $\Pi_{\phi}^{\bfH}$.

\subsection{Cayley transform for classical groups}
In this subsection, we recall the definition of the Cayley transform and its fundamental properties proved in \cite{MR3709003}.

Let $\J$ be either a general linear group or a quasi-split classical group over $F$.
We set $\mfj:=\Lie\J(F)$, where $\Lie\J$ is the Lie algebra of $\J$.
For an element $g\in J$, we say that $g$ is topologically unipotent if we have $\lim_{n\rightarrow\infty}g^{p^{n}}=1$.
We write $J_{\tu}$ for the set of topologically unipotent elements of $J$.
On the other hand, if $\J$ is a classical group associated to a $F$-vector space $V$, then we can regard $\mfj$ as a subalgebra of $\End(V)$.
Thus, for an element $X\in \mfj$, we can consider its power as a matrix.
We say that $X\in\mfj$ is topologically nilpotent if we have $\lim_{n\rightarrow\infty}X^{n}=0$.
We write $\mfj_{\tn}$ for the set of topologically nilpotent elements of $\mfj$.

Now let $(\G,\bfH)$ be one of the pairs in Section \ref{sec:pre}.1. 

\begin{defn}[Cayley transform for general linear groups]\label{defn:Cayley}
Let $\mfc$ be the map from $\mathfrak{gl}_{N}(F)_{\tn}$ to $\GL_{N}(F)_{\tu}$ defined by
\[
\mfc\colon\mathfrak{gl}_{N}(F)_{\tn}\rightarrow \GL_{N}(F)_{\tu};\quad
X\mapsto \frac{1+\frac{X}{2}}{1-\frac{X}{2}}.
\]
We call this map $\mfc$ the \textit{Cayley transform} for $\GL_{N}$.
\end{defn}

\begin{prop}\label{prop:Cayley-properties}
\begin{enumerate}
\item
The Cayley transform $\mfc$ for $\GL_{N}$ is a homeomorphism, and its inverse is given by
\[
\mfc^{-1}\colon\GL_{N}(F)_{\tu}\rightarrow\mathfrak{gl}_{N}(F)_{\tn};\quad
g\mapsto \frac{2(g-1)}{g+1}.
\]
\item
For any $A\in\GL_{N}(F)$, we have 
\[
\mfc\circ (X\mapsto -{}^{t}\!X) = (g\mapsto {}^{t}\!g^{-1})\circ\mfc
\quad\text{and}\quad
\mathrm{Int}(A)\circ\mfc=\mfc\circ\Ad(A).
\]
In particular, for any quasi-split classical group $\J$ over $F$, $\mfc$ defines a homeomorphism from $\mfj_{\tn}$ to $J_{\tu}$.
\end{enumerate}
\end{prop}

\begin{proof}
The first assertion follows from Lemma 3.2.3 in \cite{MR3709003} and an easy computation.
The second assertion is cited from Remark 3.2.4 in \cite{MR3709003}.
\end{proof}

\begin{defn}\label{defn:Cayley-classic}
Let $\mfc'$ be the map from $\mfh_{\tn}$ to $H_{\tu}$ defined by
\[
\mfc'\colon\mfh_{\tn}\rightarrow H_{\tu};\quad
X\mapsto
\begin{cases}
\mfc(\frac{X}{2})^{2}& \text{if $\bfH=\Sp_{2n}$ or $\bfH=\SO_{2n+1}$},\\
\mfc(X)^{2}& \text{if $\bfH=\SO_{2n}$ or $\bfH=\U_{E/F}(N)$}.
\end{cases}
\]
\end{defn}


\subsection{Moy--Prasad filtrations of classical groups}
In this subsection, we collect basic properties of the Moy--Prasad filtrations of parahoric subgroups of classical groups.
We follow the notations of \cite[Section 10.1]{MR3709003}.
Namely, we use the following notations.
Let $\J$ be a connected reductive group over $F$.
We denote its Bruhat-Tits building by $\B(\J,F)$.
For a point $x\in\B(\J,F)$, we have a corresponding parahoric subgroup $J_{x}$ of $J$ and its Moy--Prasad filtration $\{J_{x,r}\}_{r\in\R_{\geq0}}$ (note that we have $J_{x}=J_{x,0}$).
Similarly to these filtrations, we also have the Moy--Prasad filtration $\{\mathfrak{j}_{x,r}\}_{r\in\R}$ of the Lie algebra $\mathfrak{j}=\Lie\J(F)$.
Here we note that we use the valuation of $F$ to define these filtrations and that it differs from the original definition in \cite{MR1371680} (in \cite{MR1371680}, the valuation of the splitting field of $\J$ is used to define the Moy--Prasad filtrations).
Namely, we normalize the indices of the Moy--Prasad filtrations so that, for any uniformizer $\varpi_{F}$ of $F$, the following hold:
\[
\mathfrak{j}_{x,r+1}=\varpi_{F}\cdot\mathfrak{j}_{x,r}.
\]
We write $J_{x,r:r+}$ and $\mathfrak{j}_{x,r:r+}$ for the quotients $J_{x,r}/J_{x,r+}$ and $\mathfrak{j}_{x,r}/\mathfrak{j}_{x,r+}$, respectively (here $r+$ means $r+\varepsilon$ for a sufficiently small positive number $\varepsilon$).
We put 
\[
J_{r(+)}:=\bigcup_{x\in\B(\J,F)} J_{x,r(+)}
\quad\text{and}\quad
\mathfrak{j}_{r(+)}:=\bigcup_{x\in\B(\J,F)} \mathfrak{j}_{x,r(+)}
\]
(note that we have $J_{\tu}=J_{0+}$).

When $\J$ is either a general linear group or a quasi-split classical group over $F$, we have the following property on the Cayley transform:
\begin{prop}[{\cite[Lemma 10.2.1]{MR3709003}}]\label{prop:Cayleyhomeo}
Let $\J$ be either a general linear group or a quasi-split classical group over $F$.
Let $x\in\B(\J,F)$ and $r\in\R_{>0}$.
The Cayley transform $\mfc$ induces a homeomorphism
\[
\mfc\colon\mfj_{x,r}\rightarrow{J}_{x,r}.
\]
In particular, $\mfc$ induces an isomorphism of abelian groups
\[
\mfc\colon\mfj_{x,r:r+}\xrightarrow{\cong}{J}_{x,r:r+}.
\]
\end{prop}
Here we remark that, in \cite{MR3709003}, they treat only the cases of symplectic groups and orthogonal groups.
However, by the exactly same arguments, we can show the above proposition for quasi-split unitary groups.

We next recall the compatibility of the Moy--Prasad filtrations of general linear groups with the involution $\theta$ (here we use the notations in Section \ref{sec:pre}.1).
First, we note that the Bruhat-Tits building $\B(\G_{\theta},F)$ of $\G_{\theta}$ can be $G_{\theta}$-equivariantly identified with the $\theta$-fixed points of $\B(\G,F)$ (see \cite[Remark 10.2.2]{MR3709003}).
In the rest of this paper, we always use this identification $\B(\G_{\theta},F)\cong\B(\G,F)^{\theta}$.
Under this identification, we have the following properties:

\begin{prop}[{\cite[Lemma 10.2.3]{MR3709003}}]\label{prop:twistMP}
Let $x\in\B(\G_{\theta},F)$ and $r\in\R_{>0}$.
Then we have 
\[
({G}_{x,r})^{\theta}=G_{\theta,x,r}
\quad\text{and}\quad
({\mfg}_{x,r})^{d\theta}=\mfg_{\theta,x,r}.
\]
Moreover, we can identify $G_{\theta,x,r:r+}$ and $\mfg_{\theta,x,r:r+}$ with $({G}_{x,r:r+})^{\theta}$ and $({\mfg}_{x,r:r+})^{d\theta}$, respectively.
\end{prop}

\begin{proof}
The first assertion can be deduced from the comparison theorem of the lattice filtrations and the Moy--Prasad filtrations (\cite{MR2532460}) and Proposition \ref{prop:Cayleyhomeo} (see \cite[Remark 10.2.2]{MR3709003}).

We check the second assertion.
By the first assertion, we can identify $G_{\theta,x,r:r+}$ and $\mfg_{\theta,x,r:r+}$ as subsets of $({G}_{x,r:r+})^{\theta}$ and $({\mfg}_{x,r:r+})^{d\theta}$, respectively.
We show that $G_{\theta,x,r:r+}=({G}_{x,r:r+})^{\theta}$ and that $\mfg_{\theta,x,r:r+}=({\mfg}_{x,r:r+})^{d\theta}$.
By Proposition \ref{prop:Cayleyhomeo} and the commutativity of $\mfc$ and $\theta$ (Proposition \ref{prop:Cayley-properties}), it suffices to show only the latter equality $\mfg_{\theta,x,r:r+}=({\mfg}_{x,r:r+})^{d\theta}$.
Let $X$ be an element of $\mfg_{x,r}$ satisfying $d\theta(X+\mfg_{x,r+})=X+\mfg_{x,r+}$.
If we put $Y:=d\theta(X)-X$ and $X':=X+\frac{1}{2}Y$, then we have $Y\in\mfg_{x,r+}$ and $X'\in\mfg_{\theta,x,r}$.
Namely, the coset $X'+\mfg_{\theta,x,r+}$ of $\mfg_{\theta,x,r:r+}$ maps to $X+\mfg_{x,r+}$ under the injection $\mfg_{\theta,x,r:r+}\hookrightarrow({\mfg}_{x,r:r+})^{d\theta}$.
This completes the proof.
\end{proof}

\subsection{Depth of representations}
In this subsection, we recall the notion of \textit{depth} of representations.
Let $\J$ be a connected reductive group over $F$.

\begin{defn}\label{defn:repdepth}
For an irreducible smooth representation $\pi$ of $J$, we define its depth to be
\[
\depth(\pi)
:=
\inf \bigr\{r\in\R_{\geq0} \,\big\vert\, \pi^{J_{x, r+}}\neq0 \text{ for some } x\in\mathcal{B}(\J,F)\bigr\}\in\R_{\geq0}.
\]
\end{defn}

\begin{prop}[{\cite[Theorem 3.5]{MR1371680}}]\label{prop:depth}
For every irreducible smooth representation $\pi$ of $J$, its depth is attained by a point of $\mathcal{B}(\J,F)$.
\end{prop}

\begin{defn}\label{defn:pardepth}
For an $L$-parameter $\phi$ of $\J$, we define its depth to be 
\[
\depth(\phi)
:=
\inf \bigr\{r\in\R_{\geq0} \,\big\vert\, \text{$\phi|_{I_{F}^{r+}}$ is trivial}\bigr\}\in\R_{\geq0}.
\]
Here $I_{F}^{\bullet}$ is the ramification filtration of the inertia subgroup $I_{F}$ of $W_{F}$.
\end{defn}

In the case of general linear groups, the following \textit{depth preserving property} of the local Langlands correspondence is known:

\begin{thm}[{\cite[2.3.6]{MR2508720} and \cite[Proposition 4.5]{MR3579297}}]\label{thm:GL}
When $\J$ is $\GL_{N}$, for every $\pi\in\Pi(\GL_{N})$ and $\phi\in\Phi(\GL_{N})$ corresponding under the local Langlands correspondence for $\GL_{N}$, we have 
\[
\depth(\pi)=\depth(\phi).
\]
\end{thm}

\section{Semisimple descent for the Moy--Prasad filtrations of general linear groups}\label{sec:descent}

Let $\G$ be a general linear group over $F$.
We consider the involution $\theta$ on $\G$ as in Section \ref{sec:pre}.1.
In this section, we investigate the semisimple descent of the characteristic functions of the Moy--Prasad filtrations of parahoric subgroups of $G$.

Now let us recall the \textit{semisimple descent} of test functions supported on the topologically unipotent elements.
We define the map $\tc$ as follows:
\[
\tc\colon G\times G_{\theta} \rightarrow \w{G};\quad (g,x)\mapsto g\cdot(x\rtimes\theta)\cdot g^{-1}.
\]
Let $\mathcal{U}$ (resp.\ $\mathcal{U}_{r}$) be the image of $G\times G_{\theta,\tu}$ (resp.\ $G\times G_{\theta,r}$) under the map $\tc$.
Then the canonical inclusion
\[
G_{\theta,\tu}\hookrightarrow\mathcal{U};\, x\mapsto x\rtimes\theta
\]
induces a bijection $\Gamma(G_{\theta,\tu})\cong\Gamma(\mathcal{U})$ (see \cite[Corollary 4.0.4]{MR3709003}).
Similarly, for $r\in\R_{>0}$, we have a bijection $\Gamma(G_{\theta,r})\cong\Gamma(\mathcal{U}_{r})$.

\begin{defn}[Semisimple descent at $\theta$]
For $f\in C_{c}^{\infty}(\mathcal{U})$ and $f_{\theta}\in C_{c}^{\infty}(G_{\theta,\tu})$ (resp.\ $f\in C_{c}^{\infty}(\mathcal{U}_{r})$ and $f_{\theta}\in C_{c}^{\infty}(G_{\theta,r})$), we say that $f_{\theta}$ is a semisimple descent of $f$ if $I(f_{\theta})$ coincides with $I(f)$ as $\C$-valued functions on $\Gamma(G_{\theta,\tu})\cong\Gamma(\mathcal{U})$ (resp. $\Gamma(G_{\theta,r})\cong\Gamma(\mathcal{U}_{r})$).
Namely, for every strongly regular semisimple element $\gamma\in G_{\theta,\tu}$ (resp.\ $\gamma\in G_{\theta,r}$), we have
\[
I_{\gamma}(f_{\theta})
=
I_{\gamma\rtimes\theta}(f).
\]
\end{defn}
Here we note that, for every $\gamma\in G_{\theta,\tu}$, the centralizer $(G^{\theta})_{\gamma}$ of $\gamma$ in $G^{\theta}$ coincides with the centralizer $G_{\gamma\rtimes\theta}$ of $\gamma\rtimes\theta$ in $G$.
We use the same Haar measure on these centralizer groups in the above orbital integrals (see \cite[Definition 4.2.2]{MR3709003}).
Then we have the following:
\begin{prop}[{\cite[Lemma 10.4.2]{MR3709003}}]
As subsets of the set of $\C$-valued functions on $\Gamma(G_{\theta,\tu})\cong\Gamma(\mathcal{U})$, we have $\mathcal{I}(G_{\theta,\tu})=\mathcal{I}(\mathcal{U})$.
Similarly, we have $\mathcal{I}(G_{\theta,r})=\mathcal{I}(\mathcal{U}_{r})$.
\[
\xymatrix{
C_{c}^{\infty}(G_{\theta,\tu})\ar@{->>}[r]&\mathcal{I}(G_{\theta,\tu})\ar@{=}[d]\ar@{^{(}->}[r]&\{\Gamma(G_{\theta,\tu})\rightarrow\C\}\ar@{=}[d]\\
C_{c}^{\infty}(\mathcal{U})\ar@{->>}[r]&\mathcal{I}(\mathcal{U})\ar@{^{(}->}[r]&\{\Gamma(\mathcal{U})\rightarrow\C\}
}
\]
\end{prop}

We show a small lemma which will be needed in the next proposition:
\begin{lem}\label{lem:alcove}
We take the diagonal maximal $F$-split torus $\mathbf{T}$ in $\G$ and consider its $\theta$-fixed part $\mathbf{T}^{\theta}$, which is a maximal $F$-split torus in $\G_{\theta}$.
Then the fundamental alcove of the apartment $\mathcal{A}(\mathbf{T}^{\theta},F)$ in $\mathcal{B}(\G_{\theta},F)$ is contained in that of the apartment $\mathcal{A}(\mathbf{T},F)$ in $\mathcal{B}(\G,F)$.
\end{lem}

\begin{proof}
We check the assertion by a case-by-case computation.
See, for example, \cite[Section 10.1]{MR0327923} for a description of affine roots of classical groups.

In the case of (1), we have $\G=\GL_{2n+1}$ and $\G_{\theta}=\SO_{2n+1}$.
We identify $\mathcal{A}(\mathbf{T},F)$ with $X_{\ast}(\mathbf{T})\otimes_{\Z}\R$, and $\mathcal{A}(\mathbf{T}^{\theta},F)$ with $X_{\ast}(\mathbf{T}^{\theta})\otimes_{\Z}\R=(X_{\ast}(\mathbf{T})\otimes_{\Z}\R)^{\theta}$.
If we write $e_{i}\in X^{\ast}(\mathbf{T})$ for the $i$-th projection from $\mathbf{T}$ to $\mathbb{G}_{m}$, then $X^{\ast}(\mathbf{T}^{\theta})\otimes_{\Z}\R$ is spanned by $f_{i}:=e_{i}-e_{2n+2-i}$.
Then the sets of simple affine roots with respect to the fundamental alcoves are given by
\[
\Pi_{\G}=
\{e_{1}-e_{2},\ldots,e_{2n}-e_{2n+1}, e_{2n+1}-e_{1}+1\}\text{ and}
\]
\[
\Pi_{\G_{\theta}}=
\{f_{1}-f_{2},\ldots,f_{n-1}-f_{n}, f_{n}, -f_{1}-f_{2}+1\}.
\]
The fundamental alcoves of $\mathcal{A}(\mathbf{T},F)$ and $\mathcal{A}(\mathbf{T}^{\theta},F)$ are described as 
\[
\{x\in X_{\ast}(\mathbf{T})\otimes_{\Z}\R \mid \text{$\alpha(x)>0$ for every $\alpha\in\Pi_{\G}$}\}\text{ and}
\]
\[
\{x\in X_{\ast}(\mathbf{T}^{\theta})\otimes_{\Z}\R \mid \text{$\alpha(x)>0$ for every $\alpha\in\Pi_{\G_{\theta}}$}\}.
\]
Therefore, in order to show the assertion, we have to check the following:
for every $x\in X_{\ast}(\mathbf{T}^{\theta})\otimes_{\Z}\R$, if $\alpha(x)>0$ for every $\alpha\in\Pi_{\G_{\theta}}$, then we have $\alpha(x)>0$ for every $\alpha\in\Pi_{\G}$.
Let $x\in X_{\ast}(\mathbf{T}^{\theta})\otimes_{\Z}\R$ (note that we have $\theta(x)=x$).
Then, for $1\leq i < n$, we have 
\begin{align*}
(e_{i}-e_{i+1})(x)
&=
\frac{(e_{i}-e_{i+1})(x)+(e_{i}-e_{i+1})(\theta(x))}{2}\\
&=
\frac{(e_{i}-e_{i+1})(x)+\theta(e_{i}-e_{i+1})(x)}{2}\\
&=
\frac{(e_{i}-e_{i+1})(x)+(-e_{2n+2-i}+e_{2n+1-i})(x)}{2}
=
\frac{(f_{i}-f_{i+1})(x)}{2}.
\end{align*}
For $i=n$, we have
\[
(e_{n}-e_{n+1})(x)
=
f_{n}(x)
\]
Moreover, for $n<i\leq2n$, we have
\[
(e_{i}-e_{i+1})(x)
=
(e_{i}-e_{i+1})(\theta(x))
=
\theta(e_{i}-e_{i+1})(x)
=
(e_{2n+1-i}-e_{2n+2-i})(x).
\]
Finally, we have
\[
(e_{2n+1}-e_{1}+1)(x)
=
\bigl((-f_{1}-f_{2}+1)+e_{2}-e_{2n}\bigr)(x).
\]
Therefore, if $\alpha(x)>0$ for every $\alpha\in\Pi_{\G_{\theta}}$, then we have $\alpha(x)>0$ for every $\alpha\in\Pi_{\G}$.

In the cases of (2) and (3), we have $\G=\GL_{2n}$ and $\G_{\theta}=\Sp_{2n}$.
Then, in the same usage of notations as above, we have
\[
\Pi_{\G}=
\{e_{1}-e_{2},\ldots,e_{2n-1}-e_{2n}, e_{2n}-e_{1}+1\}\text{ and}
\]
\[
\Pi_{\G_{\theta}}=
\{f_{1}-f_{2},\ldots,f_{n-1}-f_{n}, 2f_{n}, -2f_{1}+1\},
\]
where $f_{i}=e_{i}-e_{2n+1-i}$.
For $x\in X_{\ast}(\mathbf{T}^{\theta})\otimes_{\Z}\R$, we have
\begin{align*}
(e_{i}-e_{i+1})(x)
&=
\begin{cases}
\frac{(f_{i}-f_{i+1})(x)}{2} & 1\leq i <n,\\
\frac{(f_{2n-i}-f_{2n-i+1})(x)}{2} & n+1\leq i \leq2n-1,
\end{cases}\\
(e_{n}-e_{n+1})(x)
&=
\frac{2f_{n}}{2}(x), \text{ and}\\
(e_{2n}-e_{1}+1)(x)
&=
\bigl((-2f_{1}+1)+e_{1}-e_{2n}\bigr)(x).
\end{align*}
Thus the assertion follows.

Finally, we consider the case of (4).
If $E$ is unramified over $F$, then we have
\[
\Pi_{\G}=
\{e_{1}-e_{2},\ldots,e_{N-1}-e_{N}, e_{N}-e_{1}+1\},\text{ and}
\]
\[
\Pi_{\G_{\theta}}=
\begin{cases}
\{f_{1}-f_{2},\ldots,f_{n-1}-f_{n}, 2f_{n}, -2f_{1}+1\} & N=2n,\\
\{f_{1}-f_{2},\ldots,f_{n-1}-f_{n}, f_{n}, -2f_{1}+1\} & N=2n+1,
\end{cases}
\]
where $f_{i}=e_{i}-e_{N+1-i}$.
Thus we can show the assertion by the same computation as in the cases of (2) and (3).
If $E$ is ramified over $F$, then we have
\[
\Pi_{\G}=
\biggl\{e_{1}-e_{2},\ldots,e_{N-1}-e_{N}, e_{N}-e_{1}+\frac{1}{2}\biggr\}, \text{ and}
\]
\[
\Pi_{\G_{\theta}}=
\begin{cases}
\{f_{1}-f_{2},\ldots,f_{n-1}-f_{n}, 2f_{n}, -f_{1}-f_{2}+\frac{1}{2}\} & N=2n,\\
\{f_{1}-f_{2},\ldots,f_{n-1}-f_{n}, f_{n}+\frac{1}{4}, -2f_{1}\} & N=2n+1.
\end{cases}
\]
In this case, we can express $e_{N}-e_{1}+\frac{1}{2}$ by using positive affine roots of $\Pi_{\G_{\theta}}$ and positive constants as follows:
\[
e_{N}-e_{1}+\frac{1}{2}
=-f_{1}+\frac{1}{2}
=
\begin{cases}
(-f_{1}-f_{2}+\frac{1}{2})+\frac{2f_{2}}{2} & N=2n,\\
-2f_{1}+(f_{1}+\frac{1}{4})+\frac{1}{4} & N=2n+1.
\end{cases}
\]
Thus we can show the assertion by the same way as in the previous cases.
\end{proof}

The following is the most essential proposition in this paper:

\begin{prop}[{generalization of \cite[Lemma 4.2.4 (i)]{MR3709003}}]\label{prop:GV1}
Let $x\in\B(\G_{\theta},F)$ and $r\in\R_{>0}$.
Then we have $\tc(G_{x,r},G_{\theta,x,r})=G_{x,r}\rtimes\theta$.
\end{prop}

\begin{proof}
%
We put $E_{0}=F$ if $\G=\GL_{N}$, and $E_{0}=E$ if $\G=\mathrm{Res}_{E/F}\GL_{N}$.
Let $e$ be the ramification index of the extension $E_{0}/F$.
We fix a uniformizer $\varpi$ of $E_{0}$ such that $\varpi^{e}$ belongs to $F^{\times}$ (namely, $\varpi^{e}$ is a uniformizer of $F$).

First, since $G_{\theta}$ acts on the set of alcoves of $\mathcal{B}(\G_{\theta},F)$ transitively (see, e.g., \cite[Remark 2]{MR2435422}), we may assume that $x$ is contained in the fundamental alcove of $\mathcal{B}(\G_{\theta},F)$ by taking $G_{\theta}$-conjugation.
On the other hand, the inclusion $\mathcal{B}(\G_{\theta},F)\subset\mathcal{B}(\G,{F})$ maps the fundamental alcove of $\mathcal{B}(\G_{\theta},F)$ into that of $\mathcal{B}(\G,F)$ by Lemma \ref{lem:alcove}.
Namely, we may assume that $\mfg_{x,0}$ is contained in the standard Iwahori sublattice $\mfi$ of $\mfg=\mathfrak{gl}_{N}(E_{0})$:
\[
\mfi = \begin{pmatrix}
 \mathcal{O}_{E_{0}}&&\mathcal{O}_{E_{0}}\\
 &\ddots&\\
 \mathfrak{p}_{E_{0}}&&\mathcal{O}_{E_{0}}
\end{pmatrix}.
\]
In particular, we may assume that $\mfg_{x,r}$ is contained in $\mfi_{s-1}$.
Here $\mfi_{\bullet}$ is the Moy--Prasad filtration of $\mfi$ attached to the barycenter of the fundamental alcove of $\mathcal{B}(\G,{F})$ and $s\in\Z_{>0}$ is the integer satisfying $s-1<r\leq s$.
Then we have the following chain of lattices:
\[
\mfi_{s-1+}\supset\mfg_{x,r}\supset\mfi_{s}.
\]
We note that, if we put
\[
\varphi
:=
\begin{pmatrix}
0&1&\hdots&0\\
\vdots&\ddots&\ddots&\vdots\\
0&&\ddots&1\\
\varpi&0&\hdots&0
\end{pmatrix},
\]
then we have $\mfi_{s-1+}=\varphi\cdot\mfi_{s-1}$ and $\mfi_{s}=\varpi^{e}\cdot\mfi_{s-1}=\varphi^{eN}\cdot\mfi_{s-1}$.

Now let us prove the assertion.
We follow the proof of \cite[Lemma 4.2.4 (i)]{MR3709003}.
Since the inclusion $\tc({G}_{x,r},G_{\theta,x,r})\subset{G}_{x,r}\rtimes\theta$ is clear, our task is to show the converse inclusion $\tc({G}_{x,r},G_{\theta,x,r})\supset{G}_{x,r}\rtimes\theta$.
We recall that, by Proposition \ref{prop:Cayleyhomeo}, $\mfc$ defines a homeomorphism from $\mfg_{x,r}$ to $G_{x,r}$.
Hence it is enough to show that $\tc({G}_{x,r},G_{\theta,x,r})\supset\mfc(\mfg_{x,r})\rtimes\theta$.

Since $\mfg_{x,r}$ is $\theta$-stable, we can consider the eigenspace decomposition of $\mfg_{x,r}$ with respect to $d\theta$ (note that we can always take such a decomposition since the order of $d\theta$ is $2$ and $2$ is invertible in the ring of integers $\mcO_{E_{0}}$ by the assumption that $p$ is not equal to $2$).
We denote the eigenspace with the eigenvalue $-1$ by $\mfg_{x,r}^{d\theta=-1}$.
Then we have $\mfg_{x,r}=\mfg_{\theta,x,r}\oplus\mfg_{x,r}^{d\theta=-1}$.

By the submersivity of the map $\tc$ (\cite[Lemma 4.0.6]{MR3709003}), $\tc({G}_{x,r},G_{\theta,x,r})$ is an open subset of ${G}_{x,r}\rtimes\theta=\mfc(\mfg_{x,r})\rtimes\theta$.
Combining this with the compactness of $\mfg_{\theta,x,r}$, we can take an integer $m\in\Z_{\geq0}$ such that 
$\tc({G}_{x,r},G_{\theta,x,r})\supset\mfc(\mfg_{\theta,x,r}\oplus\varpi^{m+1}\mfg_{x,r}^{d\theta=-1})\rtimes\theta$.
Indeed, for every $W\in \mfg_{\theta,x,r}$, we can take a positive integer $m_{W}$ satisfying $\tc({G}_{x,r},G_{\theta,x,r})\supset (W+\varpi^{m_{W}}\mfg_{\theta,x,r})\oplus\varpi^{m_{W}}\mfg_{x,r}^{d\theta=-1}$ by the openness of $\tc({G}_{x,r},G_{\theta,x,r})$.
Then we have 
\[
\tc({G}_{x,r},G_{\theta,x,r})
\supset
\bigcup_{W\in\mfg_{\theta,x,r}} \bigl((W+\varpi^{m_{W}}\mfg_{\theta,x,r})\oplus\varpi^{m_{W}}\mfg_{x,r}^{d\theta=-1}\bigr)\rtimes\theta
\supset
\mfg_{\theta,x,r}\rtimes\theta.
\]
On the other hand, by the compactness of $\mfg_{\theta,x,r}$, we can take a finite subset $\{W_{i}\}_{i=1}^{n}$ satisfying
\[
\bigcup_{i=1}^{n} (W_{i}+\varpi^{m_{W_{i}}}\mfg_{\theta,x,r})
\supset
\mfg_{\theta,x,r}.
\]
Thus, if we put $m+1$ to be the maximum of $\{m_{W_{i}}\}_{i=1}^{n}$, then we have
\begin{align*}
\tc({G}_{x,r},G_{\theta,x,r})
&\supset
\bigcup_{i=1}^{n} \bigl((W_{i}+\varpi^{m_{W_{i}}}\mfg_{\theta,x,r})\oplus\varpi^{m_{W_{i}}}\mfg_{x,r}^{d\theta=-1}\bigr)\rtimes\theta\\
&\supset
\bigcup_{i=1}^{n} \bigl((W_{i}+\varpi^{m_{W_{i}}}\mfg_{\theta,x,r})\oplus\varpi^{m+1}\mfg_{x,r}^{d\theta=-1}\bigr)\rtimes\theta\\
&\supset
(\mfg_{\theta,x,r}\oplus\varpi^{m+1}\mfg_{x,r}^{d\theta=-1})\rtimes\theta.
\end{align*}

From this, we will show that $\tc({G}_{x,r},G_{\theta,x,r})\supset\mfc(\mfg_{\theta,x,r}\oplus\varpi^{m}\mfg_{x,r}^{d\theta=-1})\rtimes\theta$.
If we can show this, then we get $\tc({G}_{x,r},G_{\theta,x,r})\supset\mfc(\mfg_{x,r})\rtimes\theta$ by the reverse induction on $m$.

Now we take an element $g\in\mfc(\mfg_{\theta,x,r}\oplus\varpi^{m}\mfg_{x,r}^{d\theta=-1})$.
In order to show that $\tc({G}_{x,r},G_{\theta,x,r})\ni g\rtimes\theta$, we first show the following claim:
\begin{claim*}
For any $k\in\Z_{>0}$, there exists $y\in{G}_{x,r}$ satisfying
\[
y^{-1}g\theta(y)\in\mfc(\mfg_{\theta,x,r}+\varpi^{m}\mfg_{x,r}^{k}),
\]
where
\[
\mfg_{x,r}^{k}
:=
\underbrace{\mfg_{x,r}\cdots\mfg_{x,r}}_{k}
:=
\Span_{\mcO_{E_{0}}}\{Z_{1}\cdots Z_{k}\mid Z_{1},\ldots,Z_{k}\in\mfg_{x,r}\}
\]
(note that $\mfg_{x,r}^{k}$ is an $\mathcal{O}_{E_{0}}$-submodule of $\mfg_{x,r}$ since we have $\mfg_{x,r}^{2}\subset\mfg_{x,r}$, see \cite[Lemma 10.2.3 (d)]{MR3709003}).
\end{claim*}

\begin{proof}[Proof of Claim.]
We show this claim by induction on $k$.
If $k=1$, the assertion in obvious since we have 
\[
\mfc(\mfg_{\theta,x,r}+\varpi^{m}\mfg_{x,r}^{1})
=
\mfc\bigl(\mfg_{\theta,x,r}+\varpi^{m}(\mfg_{\theta,x,r}\oplus\mfg_{x,r}^{d\theta=-1})\bigr)
=
\mfc(\mfg_{\theta,x,r}\oplus\varpi^{m}\mfg_{x,r}^{d\theta=-1})
\]
and $g$ already belongs to this set.

Next we assume the assertion for $k$, and show the assertion for $k+1$.
By the induction hypothesis (the assertion for $k$), we can take an element $y\in{G}_{x,r}$ satisfying
\[
y^{-1}g\theta(y)\in\mfc(\mfg_{\theta,x,r}+\varpi^{m}\mfg_{x,r}^{k}).
\]
We take $X_{1}\in\mfg_{\theta,x,r}$ and $X_{2}\in\mfg_{x,r}^{d\theta=-1}\cap\varpi^{m}\mfg_{x,r}^{k}$ satisfying $y^{-1}g\theta(y)=\mfc(X_{1}+X_{2})\in\mfc(\mfg_{\theta,x,r}\oplus\varpi^{m}\mfg_{x,r}^{d\theta=-1})$ (note that $\mfg_{x,r}^{k}$ is $\theta$-stable, hence we can take such $X_{2}$).
It is enough to find an element $y'\in{G}_{x,r}$ satisfying
\[
y'^{-1}y^{-1}g\theta(y)\theta(y')\in\mfc(\mfg_{\theta,x,r}+\varpi^{m}\mfg_{x,r}^{k+1}).
\]

We put $X:=X_{1}+X_{2}$ and $Y:=\frac{1}{2}X_{2}$.
We show that $y':=\mfc(Y)$ satisfies the above condition.
We first note that, for any $W\in\mfg_{\tn}$, the power series expansion of $\mfc(W)$ is given by
\[
\mfc(W)
=\frac{1+\frac{W}{2}}{1-\frac{W}{2}}
=1+W+2\cdot\biggl(\frac{W}{2}\biggr)^{2}+2\cdot\biggl(\frac{W}{2}\biggr)^{3}+\cdots.
\]
Thus, by noting that $\varpi^{m}\mfg_{x,r}^{k+1}$ is closed in $\mfg$, hence complete, we have
\[
\mfc(X)\in\mfc(X_{1})+X_{2}+\varpi^{m}\mfg_{x,r}^{k+1} 
\quad\text{and}\quad
\mfc(-Y)\in 1-Y+\varpi^{m}\mfg_{x,r}^{k+1}.
\]
On the other hand, since $\theta$ and $\mfc$ commutes (Proposition \ref{prop:Cayley-properties}) and $d\theta$ acts on $Y$ via $(-1)$-multiplication, we have $\theta(\mfc(Y))=\mfc(-Y)$.
Thus we have
\begin{align*}
y'^{-1}\cdot \mfc(X)\cdot \theta(y')
&=\mfc(-Y)\cdot \mfc(X) \cdot \mfc(-Y)\\
&\in(1-Y)\bigl(\mfc(X_{1})+X_{2}\bigr)(1-Y)+\varpi^{m}{\mfg}^{k+1}_{x,r}\\
&=\mfc(X_{1})+X_{2}-2Y+\varpi^{m}{\mfg}^{k+1}_{x,r}\\
&=\mfc(X_{1})+\varpi^{m}{\mfg}^{k+1}_{x,r}
\end{align*}
(note that we have $Y\mfc(X_{1})$ and $\mfc(X_{1})Y$ belong to $Y+\varpi^{m}{\mfg}^{k+1}_{x,r}$).
Since $X$ and $-Y$ belong to ${\mfg}_{x,r}$, their images $\mfc(X)$ and $\mfc(-Y)$ belong to $G_{x,r}$ by Proposition \ref{prop:Cayleyhomeo}.
In particular, the product $\mfc(-Y)\cdot \mfc(X) \cdot \mfc(-Y)$ lies in ${G}_{x,r}=\mfc(\mfg_{x,r})$.
If we put 
\[
\mfc(-Y)\cdot \mfc(X) \cdot \mfc(-Y)=\mfc(X_{1})+Z
\]
 for $Z\in\varpi^{m}\mfg_{x,r}^{k+1}$, then, by Proposition \ref{prop:Cayley-properties} (1), the inverse image of $\mfc(-Y)\cdot \mfc(X) \cdot \mfc(-Y)$ via $\mfc$ is given by
\begin{align*}
\mfc^{-1}\bigl(\mfc(X_{1})+Z\bigr)
&=\frac{2(\mfc(X_{1})+Z-1)}{\mfc(X_{1})+Z+1}
=\frac{\mfc(X_{1})+Z-1}{1+\frac{(\mfc(X_{1})+Z-1)}{2}}\\
&=\bigl(\mfc(X_{1})+Z-1\bigr)\biggl(1-\frac{\mfc(X_{1})+Z-1}{2}+\frac{(\mfc(X_{1})+Z-1)^{2}}{2^{2}}-\cdots\biggr)\\
&\in\bigl(\mfc(X_{1})-1\bigr)\biggl(1-\frac{\mfc(X_{1})-1}{2}+\frac{(\mfc(X_{1})-1)^{2}}{2^{2}}-\cdots\biggr)+\varpi^{m}\mfg_{x,r}^{k+1}\\
&=\mfc^{-1}\bigl(\mfc(X_{1})\bigr)+\varpi^{m}\mfg_{x,r}^{k+1}\\
&=X_{1}+\varpi^{m}\mfg_{x,r}^{k+1}.
\end{align*}
Therefore the $\theta$-conjugated element $y'^{-1}\cdot \mfc(X)\cdot \theta(y')$ belongs to $\mfc(\mfg_{\theta,x,r}+\varpi^{m}\mfg_{x,r}^{k+1})$.
This completes the proof of the claim.

\end{proof}

Now we back to the proof of Proposition \ref{prop:GV1}.
By the above claim for $k=(e+1)N$, we can find an element $y\in{G}_{x,r}$ satisfying
\[
yg\theta(y)^{-1}\in \mfc\bigl(\mfg_{\theta,x,r}+\varpi^{m}\mfg_{x,r}^{(e+1)N}\bigr).
\]
Since we have
\[
\mfg_{x,r}^{(e+1)N}
\subset\mfi_{s-1+}^{(e+1)N}
\subset\varphi^{(e+1)N}\cdot\mfi_{s-1}
=\varpi\mfi_{s}
\subset\varpi\mfg_{x,r}
\]
(note that we have $\mfi_{s-1}^{(e+1)N}\subset \mfi_{s-1}$ as $s-1\geq0$),
we can conclude
\[
yg\theta(y)^{-1}
\in \mfc\bigl(\mfg_{\theta,x,r}+\varpi^{m+1}\mfg_{x,r}\bigr)
= \mfc\bigl(\mfg_{\theta,x,r}\oplus\varpi^{m+1}\mfg_{x,r}^{d\theta=-1}\bigr).
\]
As we have $\mfc(\mfg_{\theta,x,r}\oplus\varpi^{m+1}\mfg_{x,r}^{d\theta=-1})\rtimes\theta\subset\tc({G}_{x,r},G_{\theta,x,r})$, we can conclude that $g\rtimes\theta\in\tc({G}_{x,r},G_{\theta,x,r})$.
This completes the proof.
\end{proof}

We next extend this proposition to the cosets of the Moy--Prasad filtrations of parahoric subgroups.
First recall that, for $x\in\B(\G_{\theta},F)$ and $r\in\R_{>0}$, we can canonically identify $G_{\theta,x,r:r+}$ with $(G_{x,r:r+})^{\theta}\subset G_{x,r:r+}$ (Proposition \ref{prop:twistMP}).
Note that $G_{x,r}$ acts on $G_{x,r:r+}$ via $\theta$-conjugation (in other words, $G_{x,r}$ acts on $G_{x,r:r+}\rtimes\theta$ via conjugation).
If two elements $[g_{1}], [g_{2}] \in G_{x,r:r+}$ are $\theta$-conjugate by $G_{x,r}$, then we write $[g_{1}]\sim_{\theta}[g_{2}]$.
When we regard an element $[h]\in G_{\theta,x,r:r+}$ (resp. $[g]\in G_{x,r:r+}$) as a subset of $G_{\theta,x,r}$ (resp.\ $G_{x,r}$), we denote it by $[h]_{G_{\theta}}$ (resp.\ $[g]_{G}$).

\begin{cor}\label{cor:coset}
Let $x\in\B(\G_{\theta},F)$ and $r\in\R_{>0}$.
Let $[h]\in G_{\theta,x,r:r+}$.
Then we have 
\[
\tc({G}_{x,r},[h]_{G_{\theta}})
=\bigsqcup_{\begin{subarray}{c}[g]\in{G}_{x,r:r+}\\ [g]\sim_{\theta}[h]\end{subarray}}[g]_{G}\rtimes\theta.
\]
\end{cor}

\begin{proof}
We have
\[
\tc({G}_{x,r},G_{\theta,x,r})
=\tc\biggl({G}_{x,r},\bigsqcup_{[h]\in G_{\theta,x,r:r+}} [h]_{G_{\theta}}\biggr)
=\bigcup_{[h]\in G_{\theta,x,r:r+}}\tc({G}_{x,r},[h]_{G_{\theta}}).
\]
Here note that the union in the right-hand side is in fact disjoint.
Indeed, if $[h_{1}],[h_{2}]\in G_{\theta,x,r:r+}$ satisfy $\tc({G}_{x,r},[h_{1}]_{G_{\theta}})\cap\tc({G}_{x,r},[h_{2}]_{G_{\theta}})\neq\emptyset$, then, for some element $y$ of ${G}_{x,r}$, we have $[h_{2}]=y[h_{1}]\theta(y)^{-1}$ in $G_{x,r:r+}$.
In other words, we have $[h_{2}]=[y]+[h_{1}]-[\theta(y)]$.
By this equality, we get
\[
2[h_{2}]
=[h_{2}]+\theta([h_{2}])
=[y]+[h_{1}]-\theta([y])+\theta\bigl([y]+[h_{1}]-\theta([y])\bigr)
=[h_{1}]+\theta([h_{1}])
=2[h_{1}].
\]
Since we assume that $p$ is not equal to $2$, we get $[h_{1}]=[h_{2}]$.

On the other hand, by Proposition \ref{prop:GV1}, we have
\[
\tc({G}_{x,r},G_{\theta,x,r})
={G}_{x,r}\rtimes\theta
=\bigsqcup_{[g]\in {G}_{x,r:r+}} [g]_{G}\rtimes\theta.
\]
Then, by noting that
\[
\tc({G}_{x,r},[h]_{G_{\theta}})
\subset\bigsqcup_{\begin{subarray}{c}[g]\in{G}_{x,r:r+}\\ [g]\sim_{\theta}[h]\end{subarray}} [g]_{G}\rtimes\theta,
\]
we get the assertion.
\end{proof}

Now we consider the semisimple descent of the characteristic functions of the cosets of the Moy--Prasad filtrations.
We first recall the following lemma:

\begin{lem}[{\cite[Lemma 4.2.4 (ii)]{MR3709003}}]\label{lem:GV2}
Let $C_{\theta}$ be an open compact subset of $G_{\theta,\tu}$, and $K$ a compact open subgroup of ${G}_{\tu}$.
We assume that $C_{\theta}$ is closed under conjugation by $K\cap G^{\theta}$.
Then $\vol(K\cap{G}^{\theta})^{-1}\mathbbm{1}_{C_{\theta}}\in C_{c}^{\infty}(G_{\theta})$ is a semisimple descent of $\vol(K)^{-1}\mathbbm{1}_{\tc(K,C_{\theta})}\in C_{c}^{\infty}(\w{G})$.
\end{lem}

Here we note that, in \cite{MR3709003}, this lemma is proved for pairs $(\G,\bfH)$ only in the cases of (1), (2), and (3).
However we can show the same assertion for the case of (4) by the same argument. Namely, we check the matching of Weyl discriminants for topologically unipotent elements (\cite[Lemma 4.1.3]{MR3709003} for the case of (4)) and use Kottwitz's descent lemma (\cite[Lemma 2.3]{MR2192014}).
By combining Lemma \ref{lem:GV2} with Corollary \ref{cor:coset}, we get the following consequence:

\begin{cor}\label{cor:MPdesc}
Let $x\in\B(\G_{\theta},F)$ and $r\in\R_{>0}$.
Let $[h]\in G_{\theta,x,r:r+}$.
Then $\vol(G_{\theta,x,r})^{-1}\mathbbm{1}_{[h]_{G_{\theta}}}\in C_{c}^{\infty}(G_{\theta})$ is a semisimple descent of 
\[
\vol({G}_{x,r})^{-1}
\sum_{\begin{subarray}{c}[g]\in{G}_{x,r:r+}\\ [g]\sim_{\theta}[h]\end{subarray}}\mathbbm{1}_{[g]_{G}\rtimes\theta}\in C_{c}^{\infty}(\w{G}).
\]
\end{cor}

\begin{proof}
We take $C_{\theta}$ and $K$ in Lemma \ref{lem:GV2} to be $[h]_{G_{\theta}}$ and ${G}_{x,r}$, respectively.
Then we have ${G}_{x,r}\cap{G}^{\theta}={G}_{x,r}\cap ({G}_{\theta})=G_{\theta,x,r}$ (note that we have $\G^{\theta}\neq \G_{\theta}$ only when $\G=\GL_{2n+1}$, and that, in this case, we have $\G^{\theta}=\mathrm{O}_{2n+1}=\pm\SO_{2n+1}=\pm\G_{\theta}$).
In particular, $[h]_{G_{\theta}}$ is stable under the conjugation by $K\cap G^{\theta}$, hence the assumption of Lemma \ref{lem:GV2} is satisfied.
Therefore $\vol(G_{\theta,x,r})^{-1}\mathbbm{1}_{[h]_{G_{\theta}}}\in C_{c}^{\infty}(G_{\theta})$ is a semisimple descent of $\vol({G}_{x,r})^{-1}\mathbbm{1}_{\tc({G}_{x,r},[h]_{G_{\theta}})}\in C_{c}^{\infty}(\w{G})$.
By combining this with Corollary \ref{cor:coset}, we get the assertion.
\end{proof}

\section{Evaluation of the maximum of depth in an $L$-packet}\label{sec:max}

\subsection{Semisimple descent and the transfer}
In this subsection, we recall the compatibility of the semisimple descent with the endoscopic transfer.

We first note that, for each pair $(\G,\bfH)$ in Section \ref{sec:pre}.1, the relationship between $\G_{\theta}$ and $\bfH$ is described as follows:
\begin{itemize}
\item[(1)]
In this case, $\G_{\theta}=\SO_{2n+1}$ and $\bfH=\Sp_{2n}$.
Thus $((\G_{\theta})_{\mathrm{sc}},\bfH_{\mathrm{sc}})$ can be extended to a nonstandard endoscopic triplet.

\item[(2)]
In this case, $\G_{\theta}=\Sp_{2n}$ and $\bfH=\SO_{2n+1}$.
Thus $((\G_{\theta})_{\mathrm{sc}},\bfH_{\mathrm{sc}})$ can be extended to a nonstandard endoscopic triplet.

\item[(3)]
In this case, $\G_{\theta}=\Sp_{2n}$ and $\bfH=\SO_{2n}$.
Thus $\bfH$ can be regarded as a standard endoscopic group of $\G_{\theta}$.

\item[(4)]
In this case, $\G_{\theta}=\bfH=\U_{E/F}(N)$.
\end{itemize}

In particular, in each case, we can define the notion of \textit{matching orbital integral} for test functions of $C_{c}^{\infty}(\mfg_{\theta})$ and $C_{c}^{\infty}(\mfh)$.
See \cite[Section 1.8]{MR2418405} for the cases of (1) and (2), and \cite[Section 5.5]{MR1687096} for the case of (3).
In the case of (4), we say that $f^{\mfg_{\theta}}\in C_{c}^{\infty}(\mfg_{\theta})$ and $f^{\mfh}\in C_{c}^{\infty}(\mfh)$ have matching orbital integrals if $SI(f^{\mfg_{\theta}})=SI(f^{\mfh})$ (namely, coincide as elements of $\mathcal{SI}(\mfg_{\theta})=\mathcal{SI}(\mfh)$).
Here we use the same notations as in the group case such as $SI$ and $\mathcal{SI}$.
Furthermore, in every case, similarly to Theorem \ref{thm:LSK}, we have a \textit{transfer} map from $\mathcal{I}(\mfg_{\theta})$ to $\mathcal{SI}(\mfh)$ characterized by this matching orbital integral condition.
See also \cite[Sections 6.6 and 10.5]{MR3709003}.

Then, the relationship between the three notions of matching orbital integrals, that is, the transfer from $\mathcal{I}(\w{G})$ to $\mathcal{SI}(H)$, the semisimple descent from $\mathcal{I}(\mathcal{U})$ to $\mathcal{I}(G_{\theta,\tu})$, and the transfer from $\mathcal{I}(\mfg_{\theta,\tn})$ to $\mathcal{SI}(\mfh_{\tn})$, can be stated as follows:

\begin{prop}\label{prop:ss-tran}
Let $f\in C_{c}^{\infty}(\mathcal{U})$ and $f_{\theta}\in C_{c}^{\infty}(G_{\theta,\tu})$ such that $f_{\theta}$ is a semisimple descent of $f$.
Then, for $f^{H}\in C_{c}^{\infty}(H_{\tu})$, $f^{H}$ is a transfer of $f$ if and only if $f_{\theta}\circ\mfc \in C_{c}^{\infty}(\mfg_{\theta,\tn})$ and $f^{H}\circ\mfc'\in C_{c}^{\infty}(\mfh_{\tn})$ have matching orbital integrals.
\end{prop}

This proposition can be interpreted as the commutativity of the following diagram:
\[
\xymatrix{
\mathcal{I}(G_{\theta,\tu})\ar@{=}[r]&\mathcal{I}(\mathcal{U})\ar^-{\text{transfer}}[rr]&&\mathcal{SI}(H_{\tu})\\
\mathcal{I}(\mfg_{\theta,\tu})\ar@{->}^-{\cong}_-{(\mfc^{-1})^{\ast}}[u]\ar^-{\text{transfer}}[rrr]&&&\mathcal{SI}(\mfh_{\tu})\ar@{->}^-{\cong}_-{(\mfc'^{-1})^{\ast}}[u]
}
\]
The cases of (1) and (2) of this proposition are proved in \cite[Lemma 6.6.4]{MR3709003}, and the case of (3) is proved in \cite[Lemma 7.7.2]{MR3709003}.
Finally, we can show the assertion for the case of (4) in the same manner as in these three cases.
However, for the sake of completeness, we explain the proof.
In the rest of this subsection, we focus on the case where $\G=\mathrm{Res}_{E/F}\GL_{N}$ and $\G_{\theta}=\bfH=\U_{E/F}(N)$.

First, we recall a description of the norm correspondence between $\G$ and $\bfH$ in terms of the eigenvalues of elements.
Let $\t\delta$ and $\gamma$ be semisimple elements of $\w{G}$ and $H$, respectively. 
Then, by the semisimplicity, these elements can be diagonalized.
Namely, we can take $x\in \G(\overline{F})$ and $y\in\bfH(\overline{F})$ satisfying 
\[
x \t\delta x^{-1}
= 
\bigl(\diag(t_{1},\ldots,t_{N}),\diag(s_{1},\ldots,s_{N})\bigr)\rtimes\theta
\]
and
\[
y\gamma y^{-1}
=
\diag(v_{1},\ldots,v_{N}).
\]
Here note that we have $\G_{\overline{F}}=(\mathrm{Res}_{E/F}\GL_{N,E})_{\overline{F}}\cong\GL_{N,\overline{F}}\times\GL_{N,\overline{F}}$.
Then, $\gamma$ is a norm of $\t\delta$ if and only if we have
\[
\{v_{1},\ldots,v_{N}\}
=
\biggl\{\frac{t_{1}}{s_{N}},\ldots,\frac{t_{N}}{s_{1}}\biggr\}.
\]

On the other hand, also the topological unipotency is characterized in terms of the eigenvalues.
Namely, for a semisimple element $g$ of a classical group over $F$, it is topologically unipotent if and only if we have $\mathrm{val}(\alpha(g)-1)>0$ for every eigenvalue $\alpha(g)$ of $g$.
Hence if $g\in G_{\theta,\tu}$ is a semisimple element, then, for some element $z\in \G_{\theta}(\overline{F})$, we have $zgz^{-1}=\diag(t_{1},\ldots,t_{N})$ and $\mathrm{val}(t_{i}-1)>0$ for every $1\leq i\leq N$.
Thus, if $\t\delta\in\w{G}$ is a semisimple element belonging to $\mathcal{U}$, by the above interpretation of the norm correspondence via the eigenvalues, every norm of $\t\delta$ belongs to $H_{\tu}$. 

Finally we recall that, for the pair $(\G=\mathrm{Res}_{E/F}\GL_{N}, \bfH=\U_{E/F}(N)$, the Kottwitz--Shelstad transfer factor $\Delta^{\mathrm{IV}}$ is trivial.
See, for example, \cite[1.10 Proposition]{MR2672539}.
Note that we implicitly consider the standard base change $L$-embedding from ${}^{L}\bfH$ to ${}^{L}\G$ (in the sense of Rogawski, see \cite[Section 4.7]{MR1081540} or \cite[Section 2.1]{MR3338302}), and we have $I^{-}=\emptyset$ (hence $d^{-}=0$) in the notation of \cite{MR2672539}.
In particular, every term in the formula of \cite[1.10 Proposition]{MR2672539} is trivial.

\begin{proof}[Proof of Proposition $\ref{prop:ss-tran}$ in the case of $(4)$]
We take $f\in C_{c}^{\infty}(\mathcal{U})$ and $f_{\theta}\in C_{c}^{\infty}(G_{\theta,\tu})$ such that $f_{\theta}$ is a semisimple descent of $f$.
Let $f^{H}\in C_{c}^{\infty}(H_{\tu})$.
Then, by the definition of the transfer, $f^{H}$ is a transfer of $f$ if and only if, for every strongly regular semisimple element $\gamma\in H$, we have
\[
SI_{\gamma}(f^{H})
=
\sum_{\t\delta\leftrightarrow\gamma/\sim} \Delta^{\mathrm{IV}}(\gamma,\t\delta)I_{\t\delta}(f).
\tag{$\ast$}
\]
However, by the above observation on the norm correspondence for topologically unipotent elements, if $\gamma\notin H_{\tu}$, then every $\t\delta\in\w{G}$ corresponding to $\gamma$ does not belong to $\mathcal{U}$.
Thus, since $f^{H}$ and $f$ are supported in the topologically unipotent elements, the condition $(\ast)$ is trivial for $\gamma$ such that $\gamma\notin H_{\tu}$.

Now we consider the condition $(\ast)$ for $\gamma\in H_{\tu}$.
We put $\gamma=\mfc'(Y)$ (recall that every element of $H_{\tu}$ can be written in this form since we have $\mfc(\mfh_{\tn})=H_{\tu}$ and the map $h\mapsto h^{2}$ gives a bijection from $H_{\tu}$ to itself, see \cite[Lemma 3.2.7]{MR3709003}).
For such an element $\gamma$, let us consider the index set of the sum in the right-hand side of $(\ast)$.
First, by the above description of the norm correspondence in terms of the eigenvalues, the element $\mfc(Y)\rtimes\theta$ is contained in this index set (namely, $\mfc'(Y)$ is a norm of $\mfc(Y)\rtimes\theta$).
Moreover, every other element appearing in the index set is $\G(\overline{F})$-conjugate to this element $\mfc(Y)\rtimes\theta$.
Namely, the index set is given by the $G$-conjugacy classes of stable conjugacy classes of $\mfc(Y)\rtimes\theta$ in $\mathcal{U}$.
Since the inclusion
\[
G_{\theta,\tu}\hookrightarrow\mathcal{U};\, x\mapsto x\rtimes\theta
\]
induces a bijection between the conjugacy classes and a bijection between the stable conjugacy classes (\cite[Corollary 4.0.4]{MR3709003}), we get the equality
\[
\{\t\delta\in\mathcal{U}\mid \text{$\gamma$ is a norm of $\t\delta$}\}/(\text{conjugacy})
\]
\[
=
\{\delta\in G_{\theta,\tu} \mid \text{$\delta$ is stably conjugate to $\mfc(Y)$}\}/(\text{conjugacy}).
\]

Thus, by combining this observation with the triviality of $\Delta^{\mathrm{IV}}$, the condition $(\ast)$ is equivalent to
\[
SI_{\gamma}(f^{H})
=
\sum_{\delta\sim_{\mathrm{st}}\mfc(Y)/\sim} I_{\delta\rtimes\theta}(f).
\]
However, since $\phi$ is a semisimple descent of $f$, we have $I_{\delta}(f_{\theta})=I_{\delta\rtimes\theta}(f)$.
Thus the above equality is furthermore equivalent to
\[
SI_{\gamma}(f^{H})
=
SI_{\mfc(Y)}(f_{\theta}).
\]
As we have $SI_{\gamma}(f^{H})=SI_{Y}(f^{H}\circ\mfc')$ and $SI_{\mfc(Y)}(f_{\theta})=SI_{Y}(f_{\theta}\circ\mfc)$, we can rephrase this condition as that $f_{\theta}\circ\mfc$ and $f^{H}\circ\mfc'$ have matching orbital integrals.
\end{proof}

If we assume that the residual characteristic $p$ is large enough, then we can extend this proposition to functions supported on $\mathcal{U}_{r}$, $G_{\theta,r}$, and $H_{r}$.
More precisely, for a pair $(\G,\bfH)$ in Section \ref{sec:pre}.1, if we put the condition that
\[
p>
\begin{cases}
2n & \text{the cases of (1), (2), and (3),}\\
N & \text{the case of (4)},
\end{cases}
\]
then every maximal torus in $\G$, $\G_{\theta}$, and $\bfH$ splits over a tamely ramified extension.
Then, for every $r>0$, the regions $\mathcal{U}_{r}$, $G_{\theta,r}$ and $H_{r}$ can be characterized in terms of the eigenvalues.
As a consequence, we can get the following diagram (see \cite[Remarks 10.1.5 and 10.5.1]{MR3709003} for details):
\[
\xymatrix{
\mathcal{I}(G_{\theta,r})\ar@{=}[r]&\mathcal{I}(\mathcal{U}_{r})\ar^-{\text{transfer}}[rr]&&\mathcal{SI}(H_{r})\\
\mathcal{I}(\mfg_{\theta,r})\ar@{->}^-{\cong}_-{(\mfc^{-1})^{\ast}}[u]\ar^-{\text{transfer}}[rrr]&&&\mathcal{SI}(\mfh_{r})\ar@{->}^-{\cong}_-{(\mfc'^{-1})^{\ast}}[u]
}
\]


\subsection{Character expansion and the endoscopic character relation: comparison of radii}

We next recall the \textit{homogeneity} of the characters of representations.
Let $\J$ be a connected reductive group over $F$.
For an irreducible smooth representation $\pi$ of $J$, we denote its character by $\Theta_{\pi}$.
For a nilpotent orbit $\mathcal{O}$ of $\mfj$, we write $\widehat{\mu_{\mathcal{O}}}$ for the following $J$-invariant distribution on $C_{c}^{\infty}(\mfj)$:
\[
f
\mapsto
\mu_{\mathcal{O}}(\hat{f}),
\]
where $\mu_{\mathcal{O}}$ is the orbital integral with respect to the nilpotent orbit $\mathcal{O}$ and $\hat{f}$ is the Fourier transform of $f$.
Here we do not recall the normalizations (i.e., the choices of measures) of these orbital integrals and the Fourier transform.
See, for example, Sections 3.1 and 3.4 in \cite{MR1914003} for the details.

\begin{defn}[character expansion]
Let $r\in\R_{>0}$ and $\mfc_{\J}$ be a $J$-equivariant homeomorphism from $\mathfrak{j}_{r}$ to $J_{r}$.
Let $\pi$ be an irreducible smooth representation of $J$.
We say that $\Theta_{\pi}$ has a \textit{character expansion} on $J_{r}$ with respect to $\mfc_{\J}$ if there exists a complex number $c_{\mathcal{O}}$ for each nilpotent orbit $\mathcal{O}$ of $\mfj$ such that, for every $f\in C_{c}^{\infty}(\mathfrak{j}_{r})$, the following equality holds:
\[
\Theta_{\pi}(f\circ\mfc_{\J}^{-1})
=
\sum_{\mathcal{O}\in\Nil(\mfj)}c_{\mathcal{O}}\cdot\widehat{\mu_{\mathcal{O}}}(f).
\]
In other words, as $J$-invariant distributions on $C_{c}^{\infty}(\mfj_{r})$, we have
\[
\Theta_{\pi}\circ(\mfc_{\J}^{-1})^{\ast}
=
\sum_{\mathcal{O}\in\Nil(\mfj)}c_{\mathcal{O}}\cdot\widehat{\mu_{\mathcal{O}}}.
\]
\end{defn}

The following is the \textit{homogeneity} of the characters of representations, which was established by DeBacker:

\begin{thm}[{\cite[Theorem 3.5.2]{MR1914003}}]\label{thm:hom}
Let $\bfH$ be a quasi-split classical group over $F$ as in Section $\ref{sec:pre}.1$ and $\mfc'$ the Cayley transform defined in Section $\ref{sec:pre}.4$.
We assume that the residual characteristic is large enough to satisfy hypotheses in \cite{MR1914003}.
Let $\Theta_{\pi}$ be the character of an irreducible smooth representation $\pi$ of $H$ of depth $r$.
Then $\pi$ has a character expansion on $H_{r+}$ with respect to $\mfc'$.
\end{thm}

\begin{lem}\label{lem:reschar}
We assume that
\[
p>
\begin{cases}
2n+1&\text{the cases of $(1)$ and $(3)$},\\
2n+2&\text{the case of $(2)$},\\
N+1&\text{the case of $(4)$}.
\end{cases}
\]
Then the hypotheses in \cite{MR1914003} which are needed for Theorem $\ref{thm:hom}$ are satisfied.
\end{lem}

\begin{proof}
We first note that the hypotheses of the above theorem consists of Hypotheses 2.2.2, 2.2.4, 2.2.5, 2.2.6, 2.2.8, 3.2.1, 3.4.1, 3.4.3 in \cite{MR1914003}.
Among these hypotheses, 3.2.1, 3.4.1, 3.4.3 are checked in \cite[Lemma 10.3.1]{MR3709003} for $p>2$.
Thus we consider the other hypotheses.

\begin{description}
\item[Hypothesis 2.2.4]
This hypothesis is equivalent to the condition in the assertion.

\item[Hypothesis 2.2.5]
If we define $\mathrm{exp}_{t}(X)$ to be $\sum_{i=0}^{m}\frac{X^{i}}{i!}$, then we can check that the adjoint action of $\mathrm{exp}_{t}(X)$ on $\mfh$ is given by $\sum_{i=0}^{m}\frac{\mathrm{ad}(X)^{i}}{i!}$ by an easy computation.
Furthermore, the uniqueness of such a map $\mathrm{exp}_{t}$ can be checked as follows.
To show the uniqueness of $\mathrm{exp}_{t}$, it is enough to show that the adjoint map $\mathrm{Ad}\colon H_{\mathrm{u}}\rightarrow \GL(\mfh)$ is injective, where $H_{\mathrm{u}}$ is the set of unipotent elements of $H$.
We take a unipotent element $h \in H_{\mathrm{u}}$ and assume that we have $\mathrm{Ad}(h)(Y)=hYh^{-1}$ equals $Y$ for every $Y\in\mfh$.
If we write $h=\mfc(2X)$, where $X$ is a nilpotent element of $\mfh$, then we have
\[
\frac{1+X}{1-X}\cdot Y \cdot\frac{1-X}{1+X}=Y
\]
for every $Y\in\mfh$.
However, this equality is equivalent to $[X,Y]=XY-YX=0$.
Namely, $X$ belongs to the center of $\mfh$.
Since $X$ is nilpotent, this implies that $X=0$.
Hence we have $h=1$, and the map $\mathrm{Ad}$ is injective.

\item[Hypothesis 2.2.2]
By \cite[Appendix A.2]{MR1935848}, Hypothesis 2.2.2 follows from Hypotheses 2.2.4, 2.2.5, and an assumption on the Killing form (see the third paragraph of \cite[Appendix A.2]{MR1935848}).
However, by \cite[Proposition 4.1]{MR1758228}, the assumption on the Killing form is satisfied when at least the condition in the assertion holds.

\item[Hypothesis 2.2.6]
Since we work over the field $F$ of characteristic zero, this hypothesis is always satisfied.
More precisely, the existence of $\mathfrak{sl}_{2}$-triple (so called Jacobson--Morozov's theorem) is proved in \cite[Theorem 5.3.2]{MR794307}.
Note that, in our case, we can take a representation $\rho$ in the proof of \cite[Theorem 5.3.2]{MR794307} to be the adjoint representation of $\mfh$, which is defined over $F$.
In particular, we can find a $\mathfrak{sl}_{2}$-triple in $\mfh$.
Moreover, we can lift this $\mathfrak{sl}_{2}$-triple to the group as in the manner of \cite[Section 5.5]{MR794307}.
Finally, the uniqueness of $\mathfrak{sl}_{2}$-triple (up to conjugacy) is proved in \cite[Proposition 5.5.10]{MR794307}.
We note that, in order to take an element $g\in C_{G}(e)^{0}$ (here we follow the notation in \cite[Proposition 5.5.10]{MR794307}) rationally, we have to choose ``$m$'' in the proof of \cite[Proposition 5.5.10]{MR794307} rationally.
However, by choosing $T$ and $T_{1}$ in the proof to be maximal $F$-split tori, we can take $m$ to be an element of $M(F)$ by a well-known property on the conjugacy of maximal $F$-split tori (see, e.g., \cite[Theorem 15.2.6]{MR2458469}).

\item[Hypothesis 2.2.8]
This follows from \cite[Proposition 1.6.3]{MR1653184}.
\end{description}
\end{proof}

From now on, we assume the condition on the residual characteristic in Lemma \ref{lem:reschar}.

We next consider the twisted version of the character expansion.

\begin{defn}[twisted version of the character expansion]
Let $r\in\R_{>0}$ and $\mfc$ be the Cayley transform of $\G$ defined in Section \ref{sec:pre}.4.
Let $\pi$ be a $\theta$-stable irreducible smooth representation of $G$.
Then we have the $\theta$-twisted character $\Theta_{\pi,\theta}$ of $\pi$ which is normalized by the fixed $\theta$-stable Whittaker data of $\G$.
We say that $\Theta_{\pi,\theta}$ has a \textit{character expansion} on $\mathcal{U}_{r}$ with respect to $\mfc$ if there exists a complex number $c_{\mathcal{O}}$ for each nilpotent orbit $\mathcal{O}$ of $\mfg_{\theta}$ such that, for every $f\in C_{c}^{\infty}(\mfg_{\theta,r})$, the following equality holds:
\[
\Theta_{\pi,\theta}(f\circ\mfc^{-1})
=
\sum_{\mathcal{O}\in\Nil(\mfg_{\theta})}c_{\mathcal{O}}\cdot\widehat{\mu_{\mathcal{O}}}(f).
\]
Here, we regard $f\circ\mfc^{-1}$ as an element of $\mathcal{I}(\mathcal{U}_{r})$ via the identification $\mathcal{I}(G_{\theta,r})=\mathcal{I}(\mathcal{U}_{r})$ (note that $\Theta_{\pi,\theta}$ is $G$-invariant, hence factors through $\mathcal{I}(\mathcal{U}_{r})$).
In other words, as elements of $\mathcal{I}(\mfg_{\theta,r})$, we have
\[
\Theta_{\pi,\theta}\circ(\mfc^{-1})^{\ast}
=
\sum_{\mathcal{O}\in\Nil(\mfg_{\theta})}c_{\mathcal{O}}\cdot\widehat{\mu_{\mathcal{O}}}.
\]
\end{defn}

\begin{thm}[{\cite[Corollary 12.9]{MR2306039}}]\label{thm:t-hom}
We assume that the residual characteristic is large enough $($the same assumption as that for $\G_{\theta}$ in Theorem $\ref{thm:hom}$$)$.
Let $\pi$ be a $\theta$-stable irreducible smooth representation of $G$ of depth $r$.
Then $\Theta_{\pi,\theta}$ has a character expansion on $\mathcal{U}_{r+}$ with respect to $\mfc$.
\end{thm}

From now on, we additionally assume that the residual characteristic $p$ is large enough to satisfy the assumption of this theorem.
Namely, we add the assumption that
\[
p>
\begin{cases}
2n+2&\text{the case of $(1)$},\\
2n+1&\text{the cases of $(2)$ and $(3)$},\\
N+1&\text{the case of $(4)$}.
\end{cases}
\]
Therefore, in total, we assume the following hypothesis on the residual characteristic:
\begin{hyp}\label{hyp}
The residual characteristic $p$ is greater than
\[
\begin{cases}
2n+2&\text{the cases of $(1)$ and $(2)$},\\
2n+1&\text{the case of $(3)$},\\
N+1&\text{the case of $(4)$}.
\end{cases}
\]
\end{hyp}

Now let us compare these ``radii'' of character expansions via the endoscopic character relation.
Let $(\G,\bfH)$ be one of the pairs defined in Section \ref{sec:pre}.1.
Let $\phi$ be a tempered $L$-parameter of $\bfH$.
Then, by Theorem \ref{thm:Arthur}, we get a tempered $L$-packet $\Pi_{\phi}^{\bfH}$ of $\bfH$ and an irreducible $\theta$-stable tempered representation $\pi_{\phi}^{\G}$ of $G$ corresponding to the $L$-parameter $\phi$.

\begin{lem}\label{lem:radii}
We put $r_{\bfH}:=\max\{\depth(\pi) \mid \pi \in \Pi_{\phi}^{\bfH}\}$.
Then the $\theta$-twisted character $\Theta_{\phi,\theta}^{\G}$ of $\pi_{\phi}^{\G}$ has a character expansion on $\mathcal{U}_{r_{\bfH}+}$ with respect to $\mfc$.
\end{lem}

\begin{proof}
Before we start to prove this lemma, we recall that the following diagram commutes for every $r\in\R_{>0}$ (this is obtained by taking the dual of the diagram in Section \ref{sec:max}.1):
\[
\xymatrix{
\mathcal{I}(G_{\theta,r})^{\ast}\ar@{=}[r]\ar@{->}_-{\cong}^-{((\mfc^{-1})^{\ast})^{\ast}}[d]&\mathcal{I}(\mathcal{U}_{r})^{\ast}&&\mathcal{SI}(H_{r})^{\ast}\ar_-{(\text{transfer})^{\ast}}[ll]\ar@{->}_-{\cong}^-{((\mfc'^{-1})^{\ast})^{\ast}}[d]\\
\mathcal{I}(\mfg_{\theta,r})^{\ast}&&&\mathcal{SI}(\mfh_{r})^{\ast}\ar_-{(\text{transfer})^{\ast}}[lll]
}
\]

By the definition of $r_{\bfH}$ and Theorem \ref{thm:hom}, the sum $\Theta_{\phi}^{\bfH}$ of the characters of representations belonging to $\Pi_{\phi}^{\bfH}$ has a character expansion on $H_{r_{\bfH}+}$ with respect to $\mfc'$.
Namely, as elements of $\mathcal{I}(\mfh_{r_{\bfH}+})^{\ast}$ (hence of $\mathcal{SI}(\mfh_{r_{\bfH}+})^{\ast}$), we have 
\[
\Theta_{\phi}^{\bfH}\circ(\mfc'^{-1})^{\ast}
=
\sum_{\mathcal{O}_{H}\in\Nil(\mfh)}c_{\mathcal{O}_{H}}\cdot\widehat{\mu_{\mathcal{O}_{H}}}.
\]
Similarly, if we put $r_{\G}$ to be the depth of $\pi_{\phi}^{\G}$, then $\Theta_{\phi,\theta}^{\G}$ has a character expansion on $\mathcal{U}_{r_{\G}+}$ by Theorem \ref{thm:t-hom}.
Namely, as elements of $\mathcal{I}(\mathcal{U}_{r_{\G}+})^{\ast}=\mathcal{I}(\mfg_{\theta,r_{\G}+})^{\ast}$, we have
\[
\Theta_{\phi,\theta}^{\G}\circ(\mfc^{-1})^{\ast}
=
\sum_{\mathcal{O}\in\Nil(\mfg_{\theta})}c_{\mathcal{O}}\cdot\widehat{\mu_{\mathcal{O}}}.
\tag{$\ast$}
\]

On the other hand, by the endoscopic character relation (Theorem \ref{thm:Arthur}), the pull back of $\Theta_{\phi}^{\bfH}$ via the transfer map coincides with $\Theta_{\phi,\theta}^{\G}$ as elements of $\mathcal{I}(\w{G})^{\ast}$.
Thus, by the commutativity of the above diagram, we have
\[
\Theta_{\phi,\theta}^{\G}\circ(\mfc^{-1})^{\ast}
=
(\mathrm{transfer})^{\ast}
\biggl(\sum_{\mathcal{O}_{H}\in\Nil(\mfh)}c_{\mathcal{O}_{H}}\cdot\widehat{\mu_{\mathcal{O}_{H}}}\biggr)
\]
as an element of $\mathcal{I}(\mfg_{\theta,r_{\bfH}+})^{\ast}$.
However, by the homogeneity argument (see \cite[Lemma 10.5.3]{MR3709003}), we have
\[
\sum_{\mathcal{O}\in\Nil(\mfg_{\theta})}c_{\mathcal{O}}\cdot\widehat{\mu_{\mathcal{O}}}
=
(\mathrm{transfer})^{\ast}
\biggl(\sum_{\mathcal{O}_{H}\in\Nil(\mfh)}c_{\mathcal{O}_{H}}\cdot\widehat{\mu_{\mathcal{O}_{H}}}\biggr)
\]
in $\mathcal{I}(\mfg_{\theta,0+})^{\ast}$.
In particular, the equality $(\ast)$ holds in $\mathcal{I}(\mfg_{\theta,r_{\bfH}+})^{\ast}$.
Namely, the $\theta$-twisted character $\Theta_{\phi,\theta}^{\G}$ has a character expansion on $\mathcal{U}_{r_{\bfH}+}$ with respect to $\mfc$.
\end{proof}


\subsection{Utilization of DeBacker's parametrization of nilpotent orbits}

\begin{thm}\label{thm:main}
We assume Hypothesis $\ref{hyp}$.
Let $\phi$ be a tempered $L$-parameter of $\bfH$, and $\Pi_{\phi}^{\bfH}$ the $L$-packet of $\bfH$ for $\phi$.
Then we have
\[
\max\bigl\{\depth(\pi) \,\big\vert\, \pi \in \Pi_{\phi}^{\bfH}\bigr\}
=
\depth(\phi).
\]
\end{thm}

\begin{proof}
By Theorem \ref{thm:GV} (\cite[Corollary 10.6.4]{MR3709003}), our task is to show the inequality
\[
\max\bigl\{\depth(\pi) \,\big\vert\, \pi \in \Pi_{\phi}^{\bfH}\bigr\}
\geq
\depth(\phi).
\]
In order to show this inequality, it is enough to show
\[
r
\geq
\depth(\phi)
\]
for every $r\in\R_{>0}$ which is strictly greater than $\max\bigl\{\depth(\pi) \,\big\vert\, \pi \in \Pi_{\phi}^{\bfH}\bigr\}$.
For this, we consider the converse direction of the arguments of Ganapathy-Varma in \cite[Corollary 10.6.4]{MR3709003}.

By Lemma \ref{lem:radii}, the twisted character $\Theta_{\phi,\theta}^{\G}$ of the endoscopic lift of $\Pi_{\phi}^{\bfH}$ has a local character expansion on $\mathcal{U}_{r}$:
\[
\Theta_{\phi,\theta}^{\G}\circ(\mfc^{-1})^{\ast}
=
\sum_{\mcO\in\Nil(\mfg_{\theta})}
c_{\mcO}\cdot \widehat{\mu_{\mcO}}.
\tag{$\ast$}
\]
We take a maximal (in the sense of the closure relation) nilpotent orbit $\mcO_{\star}$ of $\mfg_{\theta}$ satisfying $c_{\mcO_{\star}}\neq0$.
Here we note that there exists a nilpotent orbit $\mathcal{O}$ whose $c_{\mcO}$ is not zero (namely, the distribution in $(\ast)$ is not identically zero on $\mathcal{I}(\mathcal{U}_{r})$).
Indeed, if $(\ast)$ is identically zero on $\mathcal{I}(\mathcal{U}_{r})$, then also 
$\Theta_{\phi}^{\bfH}\circ(\mfc'^{-1})^{\ast}$ is zero on $\mathcal{SI}(\mfh_{r})$ by the injectivity of the pullback via the transfer, which is proved in the proof of \cite[Lemma 10.5.4]{MR3709003}.
However, it contradicts to the nonzeroness of $\Theta_{\phi}^{\bfH}\circ(\mfc'^{-1})^{\ast}$ in a neighborhood of the origin, which is proved in the second paragraph in the proof of \cite[Corollary 10.6.4]{MR3709003}.

By the same argument of the second paragraph of the proof of \cite[Corollary 10.6.4]{MR3709003} (namely, as a consequence of DeBacker's parametrizing result of nilpotent orbits via Bruhat-Tits theory, see \cite[Theorem 5.6.1]{MR1935848} and also \cite[Section 2.5]{MR1914003}), for this nilpotent orbit $\mcO_{\star}$, we can take a point $x\in\mathcal{B}(\G_{\theta},F)$ and an element $X\in\mfg_{\theta,x,-r}$ satisfying the following conditions:
\begin{itemize}
\item we have $X\in\mcO_{\star}$, and
\item if a nilpotent orbit $\mcO$ meets $X+\mfg_{\theta,x,-r+}$, then we have $\mcO_{\star}\subset\overline{\mcO}$.
\end{itemize}
Here we remark that, in order to use DeBacker's parametrization, we have to put some assumptions on the residual characteristic.
However, it is the same as the assumption used in Theorem \ref{thm:t-hom}.
Hence we do not have to add further assumptions on the residual characteristic.

As in the third paragraph of the proof of \cite[Corollary 10.6.4]{MR3709003}, we define a homomorphism $\chi_{X}$ from $\mfg_{\theta,x,r}$ to $\C^{\times}$ to be
\[
Y \mapsto \psi_{F}\bigl(\tr(-XY)\bigr),
\]
where $\psi_{F}$ is a nontrivial additive character of $F$ of level zero.
Then, since $X$ belongs to $\mfg_{\theta,x,-r}$, this homomorphism $\chi_{X}$ is $\mfg_{\theta,x,r+}$-invariant.
Hence, by composing the inverse of the Cayley transform isomorphism $\mfg_{\theta,x,r:r+}\cong G_{\theta,x,r:r+}$ (Proposition \ref{prop:Cayleyhomeo}) and considering the zero extension, we can regard $\chi_{X}\circ\mfc^{-1}$ as an element of $C_{c}^{\infty}(G_{\theta,x,r})$ which is bi-$G_{\theta,x,r+}$-invariant.

By Corollary \ref{cor:MPdesc}, there exists a bi-$G_{x,r+}$-invariant test function $f$ of $C_{c}^{\infty}(G_{x,r})$ such that $\chi_{X}\circ\mfc^{-1}$ is a semisimple descent of $f$.
Since $G_{x,r+}$ is $\theta$-stable and $f$ is bi-$G_{x,r+}$-invariant, we have
\[
\Theta_{\phi,\theta}^{\G}(f)
=
\tr\Bigl(\pi_{\phi}^{\G}(f)\circ I_{\theta} \,\Big\vert\, (\pi_{\phi}^{\G})^{G_{x,r+}} \Bigr)
\]
by the definition of the $\theta$-twisted character distribution.
Here, $I_{\theta}$ is an intertwiner $I_{\theta}\colon\pi_{\phi}^{\G}\cong(\pi_{\phi}^{\G})^{\theta}$ normalized by using the fixed $\theta$-stable Whittaker data of $\G$.
Thus, if we can show the non-vanishing of $\Theta_{\phi,\theta}^{\G}(f)$, then the depth of $\pi_{\phi}^{\G}$ (hence the depth of $\phi$, by Theorem \ref{thm:GL}) is bounded by $r$ and the proof is completed.

By the local character expansion $(\ast)$, we have
\[
\Theta_{\phi,\theta}^{\G}(f)
=
\sum_{\mcO\in\Nil(\mfg_{\theta})}
c_{\mcO}\cdot \widehat{\mu_{\mcO}}(\chi_{X}).
\]
By noting that the Fourier transform of $\chi_{X}$ on $\mfg_{\theta}$ with respect to the non-degenerate bilinear form 
\[
\mfg_{\theta} \times \mfg_{\theta} \rightarrow \C^{\times};\quad (Y_{1},Y_{2}) \mapsto \psi_{F}\bigl(\tr(Y_{1}Y_{2})\bigr)
\]
is given by $\vol(\mfg_{\theta,x,r})\cdot\mathbbm{1}_{X+\mfg_{\theta,x,-r+}}$, we have
\begin{align*}
\sum_{\mcO\in\Nil(\mfg_{\theta})}
c_{\mcO}\cdot \widehat{\mu_{\mcO}}(\chi_{X})
&=
\sum_{\mcO\in\Nil(\mfg_{\theta})}
c_{\mcO}\cdot \mu_{\mcO}(\widehat{\chi_{X}})\\
&=
\vol(\mfg_{\theta,x,r})
\sum_{\mcO\in\Nil(\mfg_{\theta})}
c_{\mcO}\cdot \mu_{\mcO}(\mathbbm{1}_{X+\mfg_{\theta,x,-r+}}).
\end{align*}
Furthermore, by the properties of $\mcO_{\star}$ and $X$, we can compute the right-hand side as follows:
\begin{align*}
\vol(\mfg_{\theta,x,r})
\sum_{\mcO\in\Nil(\mfg_{\theta})}
c_{\mcO}\cdot \mu_{\mcO}(\mathbbm{1}_{X+\mfg_{\theta,x,-r+}})
&=
\vol(\mfg_{\theta,x,r})
\sum_{\begin{subarray}{c}\mcO\in\Nil(\mfg_{\theta})\\\mcO_{\star}\subset\overline{\mcO} \end{subarray}}
c_{\mcO}\cdot \mu_{\mcO}(\mathbbm{1}_{X+\mfg_{\theta,x,-r+}})\\
&=
\vol(\mfg_{\theta,x,r})c_{\mcO_{\star}}\cdot \mu_{\mcO_{\star}}(\mathbbm{1}_{X+\mfg_{\theta,x,-r+}}).
\end{align*}
In particular, this is not equal to zero.
This completes the proof.
\end{proof}

We finally comment on the depth of nontempered $L$-packets.
In Theorem \ref{thm:Arthur}, the local Langlands correspondence is stated only for tempered $L$-packets, and we do not have the endoscopic character relation for nontempered $L$-packets.
However, by using the theory of \textit{Langlands classification}, we can extend the local Langlands correspondence for tempered representations (Theorem \ref{thm:Arthur}) to nontempered representations as follows.

Let $\phi$ be an $L$-parameter of $\bfH$.
If we regard $\phi$ as a representation of $W_{F}\times\SL_{2}(\C)$ by composing it with the $L$-embedding from ${}^{L}\bfH$ to ${}^{L}\G$, we can decompose $\phi$ into a direct sum of representations:
\[
\phi
=
\bigoplus_{i=1}^{k}(\phi_{i}\otimes|\cdot|_{F}^{r_{i}})
\oplus\phi_{0}
\oplus\bigoplus_{i=1}^{k}(\phi_{i}\otimes|\cdot|_{F}^{r_{i}})^{\vee}.
\]
Here,
\begin{itemize}
\item
$\phi_{i}$ is a tempered $L$-parameter of $\GL_{N_{i}}$, 
\item
$\phi_{0}$ is a tempered $L$-parameter of a smaller quasi-split classical group $\bfH_{0}$ of the same type as $\bfH$ (if we put the size of $\widehat{\bfH_{0}}$ to be $N_{0}$, then we have $N=N_{0}+\sum_{i=1}^{k}2N_{i}$), 
\item
$r_{i}$'s are real numbers satisfying $r_{1}>\cdots>r_{k}\geq0$, and
\item
$|\cdot|_{F}$ is the character of the Weil group $W_{F}$ corresponding to the absolute value of $F^{\times}$ under the local class field theory for $F$.
\end{itemize}
In this situation, we can regard $\GL_{N_{1}}\times\cdots\times\GL_{N_{k}}\times\bfH_{0}$ as a Levi subgroup of $\bfH$.
Let $\mathbf{P}$ be a parabolic subgroup of $\bfH$ having $\GL_{N_{1}}\times\cdots\times\GL_{N_{k}}\times\bfH_{0}$ as its Levi subgroup.
Then, by the theory of Langlands classification, for
\begin{itemize}
\item
irreducible tempered representations $\pi_{i}$ of $\GL_{N_{i}}(F)$ corresponding to $\phi_{i}$ under the local Langlands correspondence for $\GL_{N_{i}}$, and
\item
every member $\pi_{0}$ of the tempered $L$-packet $\Pi_{\phi_{0}}^{\bfH_{0}}$ of $\bfH_{0}$ for $\phi_{0}$,
\end{itemize}
the normalized parabolic induction
\[
\nInd_{P}^{H} \bigl((\pi_{1}\otimes|\det(\cdot)|^{r_{1}})\boxtimes\cdots\boxtimes(\pi_{k}\otimes|\det(\cdot)|^{r_{k}})\boxtimes\pi_{0}\bigr)
\]
has a unique irreducible quotient.
We define the $L$-packet $\Pi_{\phi}^{\bfH}$ for $\phi$ to be the set of irreducible smooth representations of $H$ obtained as such irreducible quotients.

Now we recall the following compatibility of the parabolic induction and the depth of representations:
\begin{prop}[{\cite[Theorem 5.2]{MR1371680}}]\label{prop:parab}
Let $\J$ be a connected reductive group over $F$ and $\mathbf{P}_{\J}$ a parabolic subgroup of $\J$ with a Levi decomposition $\mathbf{P}_{\J}=\mathbf{M}_{\J}\mathbf{N}_{\J}$.
Let $\rho$ be an irreducible smooth representation of $M_{J}$ and $\pi$ an irreducible subquotient of $\nInd_{P_{J}}^{J}\rho$.
Then we have $\depth(\pi)=\depth(\rho)$.
\end{prop}

By using this, we can extend Theorem \ref{thm:main} to nontempered representations:
\begin{thm}\label{thm:main-non-temp}
We assume Hypothesis $\ref{hyp}$.
Let $\phi$ be an $L$-parameter of $\bfH$, and $\Pi_{\phi}^{\bfH}$ the $L$-packet of $\bfH$ for $\phi$.
Then we have
\[
\max\bigl\{\depth(\pi) \,\big\vert\, \pi \in \Pi_{\phi}^{\bfH}\bigr\}
=
\depth(\phi).
\]
\end{thm}

\begin{proof}
We consider the direct sum decomposition of $\phi$ using tempered $L$-parameters:
\[
\phi
=
\bigoplus_{i=1}^{k}(\phi_{i}\otimes|\cdot|_{F}^{r_{i}})
\oplus\phi_{0}
\oplus\bigoplus_{i=1}^{k}(\phi_{i}\otimes|\cdot|_{F}^{r_{i}})^{\vee}.
\]
Then, by Proposition \ref{prop:parab} and the construction of nontempered $L$-packets, for every member $\pi$ of $\Pi_{\phi}^{\bfH}$, we have
\begin{align*}
\depth(\pi)
&=
\depth\bigl((\pi_{1}\otimes|\det(\cdot)|^{r_{1}})\boxtimes\cdots\boxtimes(\pi_{k}\otimes|\det(\cdot)|^{r_{k}})\boxtimes\pi_{0}\bigr)\\
&=
\max\{\depth(\pi_{1}),\ldots,\depth(\pi_{k}),\depth(\pi_{0})\}
\end{align*}
(note that the determinant twist does not change the depth).
Hence we have
\[
\max\bigl\{\depth(\pi) \,\big\vert\, \pi \in \Pi_{\phi}^{\bfH}\bigr\}
=
\max\bigl\{\depth(\pi_{1}),\ldots,\depth(\pi_{k}),\depth(\pi_{0}) \,\big\vert\, \pi_{0} \in \Pi_{\phi_{0}}^{\bfH_{0}}\bigr\}.
\]
Therefore, by using Theorems \ref{thm:GL} and \ref{thm:main}, we get
\[
\max\bigl\{\depth(\pi) \,\big\vert\, \pi \in \Pi_{\phi}^{\bfH}\bigr\}
=
\max\{\depth(\phi_{1}),\ldots,\depth(\phi_{k}),\depth(\phi_{0})\}.
\]
The right-hand side equals $\depth(\phi)$.
\end{proof}

\section{Evaluation of the minimum of depth in an $L$-packet for unitary groups}\label{sec:min-unitary}
In this section, we consider the case where $\G:=\mathrm{Res}_{E/F}\GL_{N}$ for a quadratic extension $E/F$, and $\bfH$ is the quasi-split unitary group with respect to $E/F$ in $N$ variables.
Recall that, for these groups, we have $\G_{\theta}=\bfH$.

By combining Corollary \ref{cor:MPdesc} with Proposition \ref{prop:ss-tran}, we get the following:
\begin{thm}\label{thm:MPFL}
Let $x\in\mathcal{B}(\bfH,F)=\mathcal{B}(\G_{\theta},F)$ and $r\in\R_{>0}$.
Then $\vol(H_{x,r})^{-1}\mathbbm{1}_{H_{x,r}}\in C_{c}^{\infty}(H)$ is a transfer of $\vol(G_{x,r})^{-1}\mathbbm{1}_{G_{x,r}\rtimes\theta}\in C_{c}^{\infty}(\w{G})$.
\end{thm}

\begin{proof}
By Corollary \ref{cor:MPdesc}, $\vol(H_{x,r})^{-1}\mathbbm{1}_{H_{x,r}}\in C_{c}^{\infty}(H)$ is a semisimple descent of $\vol(G_{x,r})^{-1}\mathbbm{1}_{G_{x,r}\rtimes\theta}\in C_{c}^{\infty}(\w{G})$.
Thus, by Proposition \ref{prop:ss-tran}, in order to show the assertion, it is enough to check that $\vol(H_{x,r})^{-1}\mathbbm{1}_{H_{x,r}}\circ\mfc$ and $\vol(H_{x,r})^{-1}\mathbbm{1}_{H_{x,r}}\circ\mfc'$ have matching orbital integrals.
However, since we assume that the residual characteristic $p$ is not equal to $2$, the map $x\mapsto x^{2}$ from $H_{\tu}$ to itself is bijective and induces a bijection on $H_{x,r}$ (see, e.g., \cite[Lemma 3.2.7]{MR3709003} for a proof).
Namely, we have 
\[
\vol(H_{x,r})^{-1}\mathbbm{1}_{H_{x,r}}\circ\mfc
=\vol(H_{x,r})^{-1}\mathbbm{1}_{H_{x,r}}\circ\mfc'
=\vol(H_{x,r})^{-1}\mathbbm{1}_{\mfh_{x,r}}.
\]
In particular, they have matching orbital integrals.
\end{proof}

\begin{prop}\label{prop:main2}
Let $\phi$ be a tempered $L$-parameter of $\bfH$, and $\Pi_{\phi}^{\bfH}$ the $L$-packet of $\bfH$ for $\phi$.
Then we have
\[
\min\bigl\{\depth(\pi) \,\big\vert\, \pi \in \Pi_{\phi}^{\bfH}\bigr\}
\geq
\depth(\phi).
\]
\end{prop}

\begin{proof}
Let $\pi$ be a member of $\Pi_{\phi}^{\bfH}$ having the minimal depth in $\Pi_{\phi}^{\bfH}$, and we put $r$ to be the depth of $\pi$.
Then, for a point $x\in\B(\bfH,F)$, $\pi$ has a non-zero $H_{x,r+}$-fixed vector by Proposition \ref{prop:depth}.
Thus we have
\[
\Theta_{\pi}(\mathbbm{1}_{H_{x,r+}})
=
\tr\Bigl(\pi(\mathbbm{1}_{H_{x,r+}})\,\Big\vert\, \pi^{H_{x,r+}} \Bigr)
=\dim(\pi^{H_{x,r+}})>0.
\]
Moreover, for other member $\pi'$ of $\Pi_{\phi}^{\bfH}$, we have
\[
\Theta_{\pi'}(\mathbbm{1}_{H_{x,r+}})
=\dim({\pi'}^{H_{x,r+}})\geq0.
\]
Therefore we can conclude that $\Theta_{\phi}^{\bfH}(\mathbbm{1}_{H_{x,r+}})$ is not zero.

On the other hand, by Theorem \ref{thm:MPFL}, $\vol(H_{x,r})^{-1}\mathbbm{1}_{H_{x,r}}\in C_{c}^{\infty}(H)$ is a transfer of $\vol(G_{x,r})^{-1}\mathbbm{1}_{G_{x,r}\rtimes\theta}\in C_{c}^{\infty}(\w{G})$.
Therefore, by using the endoscopic character relation (Theorem \ref{thm:Arthur}) for $\pi_{\phi}^{\G}$, which is the endoscopic lift of $\Pi_{\phi}^{\bfH}$, and $\Pi_{\phi}^{\bfH}$ to these functions, we get
\[
\vol(G_{x,r})^{-1}\cdot\Theta_{\phi,\theta}^{\G}(\mathbbm{1}_{G_{x,r+}\rtimes\theta})
=
\vol(H_{x,r})^{-1}\cdot\Theta_{\phi}^{\bfH}(\mathbbm{1}_{H_{x,r+}})
\neq0.
\]
Since we have 
\[
\Theta_{\phi,\theta}^{\G}(\mathbbm{1}_{G_{x,r+}\rtimes\theta})
=
\tr\Bigl(\pi_{\phi}^{\G}(\mathbbm{1}_{G_{x,r+}})\circ I_{\theta} \,\Big\vert\, (\pi_{\phi}^{\G})^{G_{x,r+}} \Bigr),
\]
where $I_{\theta}$ is an isomorphism $I_{\theta}\colon\pi_{\phi}^{\G}\cong(\pi_{\phi}^{\G})^{\theta}$ normalized by the fixed Whittaker data, in particular $(\pi_{\phi}^{\G})^{G_{x,r+}}$ is not zero.
Therefore the depth of $\pi_{\phi}^{\G}$ is not greater than $r$.
As we have $\depth(\pi_{\phi}^{\G})=\phi$ by Theorem \ref{thm:GL}, we get the assertion.
\end{proof}

\begin{rem}
In the above proof, we only have to assume that $p$ is not equal to $2$.
\end{rem}

\begin{rem}
We cannot deduce the converse inequality by swapping the roles of $\G$ and $\bfH$, since, a priori, it is not clear whether we can take a point $x$ of $\mathcal{B}(\G_{\theta})$ satisfying $(\pi_{\phi}^{\G})^{G_{x,r+}}\neq0$.
Furthermore, even if such a point $x$ exists, the non-vanishing of $(\pi_{\phi}^{\G})^{G_{x,r+}}$ does not necessarily imply the non-vanishing of $\Theta_{\phi,\theta}^{\G}(\mathbbm{1}_{G_{x,r+}\rtimes\theta})$.
\end{rem}

\begin{thm}\label{thm:main2}
We assume Hypothesis $\ref{hyp}$.
Let $\phi$ be a tempered $L$-parameter of $\bfH$, and $\Pi_{\phi}^{\bfH}$ the $L$-packet of $\bfH$ for $\phi$.
Then, for every $\pi\in\Pi_{\phi}^{\bfH}$, we have
\[
\depth(\pi)=\depth(\phi).
\]
\end{thm}

\begin{proof}
By combining Proposition \ref{prop:main2} and Theorem \ref{thm:GV} (\cite[Corollary 10.6.4]{MR3709003}), we get
\[
\min\bigl\{\depth(\pi) \,\big\vert\, \pi \in \Pi_{\phi}^{\bfH}\bigr\}
\geq
\depth(\phi)
\geq
\max\bigl\{\depth(\pi) \,\big\vert\, \pi \in \Pi_{\phi}^{\bfH}\bigr\}.
\]
Thus $\depth(\pi)$ is constant on $\Pi_{\phi}^{\bfH}$ and equal to $\depth(\phi)$.
\end{proof}

By the same argument as in the proof of Theorem \ref{thm:main-non-temp}, we can extend this result to nontempered $L$-parameters:

\begin{thm}\label{thm:main2-non-temp}
We assume Hypothesis $\ref{hyp}$.
Let $\phi$ be an $L$-parameter of $\bfH$, and $\Pi_{\phi}^{\bfH}$ the $L$-packet of $\bfH$ for $\phi$.
Then, for every $\pi\in\Pi_{\phi}^{\bfH}$, we have
\[
\depth(\pi)=\depth(\phi).
\]
\end{thm}

\end{document}